\documentclass[11pt,a4paper]{article}
\usepackage[utf8]{inputenc}
\usepackage{amsmath}
\usepackage{amsfonts}
\usepackage{amssymb}
\usepackage{amsthm}
\usepackage{enumerate}
\PassOptionsToPackage{normalem}{ulem}
\usepackage{ulem}
\usepackage[unicode=true,colorlinks,citecolor=linkcolor,linkcolor=linkcolor,urlcolor=urlcolor]{hyperref}
\usepackage[left=3cm,right=3cm,top=2.5cm,bottom=2.5cm]{geometry}
\usepackage{array}
\usepackage{cases}

\usepackage{color}
\usepackage{dsfont}

\theoremstyle{plain}
\newtheorem{thm}{\protect Theorem}[section]
\newtheorem{prop}[thm]{\protect Proposition}
\newtheorem{defn}[thm]{\protect Definition}
\newtheorem{representation}[thm]{\protect Representation}
\newtheorem{lem}[thm]{\protect Lemma}
\newtheorem{rem}[thm]{\protect Remark}

\newtheorem{cor}[thm]{\protect Corollary}

\newcommand{\card}{{\rm Card}\,}
\newcommand{\nsp}{\vspace{-0.1cm}}

\definecolor{linkcolor}{rgb}{0,0,0.502}
\definecolor{urlcolor}{rgb}{1,0,0}

\begin{document}

\title{\textsc{Mean-field limit of generalized \\
Hawkes processes}}
\author{Julien Chevallier\footnote{Corresponding author: {\sf{e-mail: \href{mailto:julien.chevallier@unice.fr}{julien.chevallier@unice.fr}}}}\\
Univ. Nice Sophia Antipolis, CNRS, LJAD, UMR 7351, 06100 Nice, France.
}
\date{}

\maketitle

\begin{abstract}
We generalize multivariate Hawkes processes mainly by including a dependence with respect to the age of the process, i.e. the delay since the last point. 

Within this class, we investigate the limit behaviour, when $n$ goes to infinity, of a system of $n$ mean-field interacting age-dependent Hawkes processes. We prove that such a system can be approximated by independent and identically distributed age dependent point processes interacting with their own mean intensity. This result generalizes the study performed by Delattre, Fournier and Hoffmann in \cite{delattre2016}.

In continuity with \cite{chevallier2015microscopic}, the second goal of this paper is to give a proper link between these generalized Hawkes processes as microscopic models of individual neurons and the age-structured system of partial differential equations introduced by Pakdaman, Perthame and Salort in \cite{pakdaman2010dynamics} as macroscopic model of neurons.
\end{abstract}

\textit{Keywords}: Hawkes process,  mean-field approximation, interacting particle systems, renewal equation, neural network

\smallskip
\textit{Mathematical Subject Classification}: 60G55, 60F05, 60G57, 92B20

\section{Introduction}
In the recent years, the self-exciting point process known as the Hawkes process \cite{hawkes_1971} has been used in very diverse areas. First introduced to model earthquake replicas \cite{kagan_2010} or \cite{ogata_1998} (ETAS model), it has been used in criminology to model burglary \cite{mohler_2011}, in genomic data analysis to model occurrences of genes \cite{gusto_2005,reynaud_2010}, in social networks analysis to model viewing or popularity \cite{bao_2015,crane2008robust}, as well as in finance \cite{bacry_2012,CLT_Hawkes_lineaire}. We refer to \cite{liniger2009multivariate} or \cite{zhu7531nonlinear} for more extensive reviews on applications of Hawkes processes.
A univariate (nonlinear) Hawkes process is a point process $N$ admitting a stochastic intensity of the form
\begin{equation}\label{eq:def:Hawkes:univarie}
\lambda_{t}= \Phi\left( \int_{0}^{t-} h(t-z) N(dz) \right),
\end{equation}
where $\Phi:\mathbb{R}\rightarrow\mathbb{R}_{+}$ is called the intensity function, $h:\mathbb{R}_{+}\rightarrow\mathbb{R}$ is called the self-interaction function (also called exciting function or kernel function in the literature) and $N(dz)$ denotes the point measure associated with $N$. We refer to \cite{torrisi2016gaussian,torrisi2016poisson,zhu_2013,zhu2014process,zhu2015large} for recent papers dealing with nonlinear Hawkes process.

Such a form of the intensity is motivated by practical cases where all the previous points of the process may impact the rate of appearance of a new point. The influence of the past points is formulated in terms of the delay between those past occurrences and the present time, through the weight function $h$.
In the natural framework where $h$ is non-negative and $\Phi$ increasing, this choice of interaction models an excitatory phenomenon: each time the process has a jump, it excites itself in the sense that it increases its intensity and thus the probability of finding a new point. A classical case is the \emph{linear} Hawkes process for which $h$ is non-negative and $\Phi(x) = \mu + x$ where $\mu$ is a positive constant called the spontaneous rate.
Note however that Hawkes processes can also describe inhibitory phenomena. For example, the function $h$ may take negative values, $\Phi$ being the positive part modulo the spontaneous rate $\mu$, i.e. $\Phi(x)=\max(0,\mu+x)$.\\

Multivariate Hawkes processes consist of multivariate point processes $(N^{1},\dots ,N^{n})$ whose intensities are respectively given  for $i=1,\dots ,n$ by
\begin{equation}\label{eq:def:Hawkes:multivarie}
\lambda^{i}_{t}= \Phi_{i}\left( \sum_{j=1}^{n} \int_{0}^{t-} h_{j\to i}(t-z) N^{j}(dz) \right),
\end{equation}
where $\Phi_{i}:\mathbb{R}\rightarrow\mathbb{R}_{+}$ is the intensity function associated with the particle $i$ and $h_{j\to i}$ is the \emph{interaction function} describing the influence of each point of $N^{j}$ on the appearance of a new point onto $N^{i}$, via its intensity $\lambda^{i}$.

When the number of interacting particles is huge (as, for instance, financial or social networks agents), one may be willing to let the number of particles goes to infinity.
This is especially true for multivariate Hawkes processes subject to mean-field interactions. In such a case, we may indeed expect propagation of chaos, namely the particles are expected to become asymptotically independent, provided that they start from independent and identically distributed (i.i.d.) initial conditions and submitted to i.i.d. sources of noise. Mean-field type interactions involve some homogeneity and some symmetry through coefficients that depend upon the empirical measure of the processes: In the limit regime, the coefficients depend upon the common asymptotic distribution of the particles, which satisfies nonlinear dynamics, sometimes called of McKean-Vlasov type.

The study of mean-field situations for Hawkes processes was initiated by Delattre et al. \cite{delattre2016} 
by considering the following form of intensity
\begin{equation}\label{eq:intensity:HDHP}
\lambda^{i}_{t}= \Phi\left( \frac{1}{n} \sum_{j=1}^{n} \int_{0}^{t-}  h(t-z) N^{j}(dz) \right),
\end{equation}
where, in comparison with \eqref{eq:def:Hawkes:multivarie}, all the $\Phi_{i}$'s and the $h_{j\to i}$'s are the same.
In particular, it is shown in \cite{delattre2016} that mean-field interacting Hawkes processes are well approximated, when the size of the network $n$ goes to infinity, by i.i.d. Poisson processes of the McKean-Vlasov type in the sense that their intensity is given by the following implicit formula $\overline{\lambda}(t)=\Phi( \int_{0}^{t}  h(t-z) \overline{\lambda}(z) dz )$.

In the present article, a generalized version of Hawkes processes with mean-field interactions, namely Age Dependent Random Hawkes Processes (ADRHP for short), is studied.
For any point process $N$, we call \emph{predictable age process} associated with $N$ the predictable process $(S_{t-})_{t\geq 0}$ given by
\begin{equation*}
S_{t-}= t-\sup\{T\in N,\, T<t\}, \quad \text{ for all } t>0,
\end{equation*}
and extended by continuity in $t=0$. In particular, its value in $t=0$ is entirely determined by $N\cap\mathbb{R}_-$ and is well-defined as soon as there is a point therein. 
In comparison with the standard mean-field type Hawkes processes studied in \cite{delattre2016} we assume here that the intensity function $\Phi$ in \eqref{eq:intensity:HDHP} (which is denoted by $\Psi$ to avoid confusion) may also depend on the predictable age process $(S^{i}_{t-})_{t\geq 0}$ associated with the point process $N^{i}$, like for instance
\begin{equation}\label{eq:intensity:ADHP}
\lambda^{i}_{t}= \Psi\left( S^{i}_{t-}, \frac{1}{n} \sum_{j=1}^{n} \int_{0}^{t-}  h(t-z) N^{j}(dz) \right).
\end{equation}
This more general choice for the intensity makes the main difference with \cite{delattre2016}, where the intensity is assumed to be of the simpler form \eqref{eq:intensity:HDHP} only. We then show that, instead of Poisson processes of the McKean-Vlasov type, the limit processes associated with mean-field interacting age-dependent Hawkes processes are point processes of the McKean-Vlasov type whose stochastic intensity not only depends on the time but also on the age. More precisely, for the toy example \eqref{eq:intensity:ADHP}, the intensity of the limit process $\overline{N}$ would be given by the following implicit formula $\overline{\lambda}_{t}=\Psi( \overline{S}_{t-}, \int_{0}^{t}  h(t-z) \mathbb{E}\left[ \overline{\lambda}_{z} \right] dz )$ where $(\overline{S}_{t-})_{t\geq 0}$ is the predictable age process associated with $\overline{N}$.\\

Part of our analysis finds its motivation in the use of Hawkes processes for the modelling in neuroscience.
First of all, at a microscopic scale, Hawkes processes are commonly used in theoretical studies \cite{Chornoboy1988,hansen2015lasso,pillow2008spatio,reynaud2014goodness}
to describe the time occurrences of the action potentials of different neurons. These action potentials are associated with brutal changes of the membrane potential, called \emph{spikes} in the rest of the article. 
The motivation for using Hawkes process is well-understood and linked with the \emph{synaptic integration} phenomenon: the interaction functions $h_{j \rightarrow i}$ describe the fact that, whenever a neuron spikes, the membrane potential of the other neurons in the network (and thus their firing rate as well) may change.
In that sense, the $L^1$ norm of the interaction function $h_{j \rightarrow i}$, for $j \not = i$, is the analogue of the synaptic weight of neuron $j$ over neuron $i$, that is the strength of the influence of neuron $j$ over neuron $i$ through their synaptic connection. For example, if one considers $h_{j\to i}=\alpha_{j\to i}h$ for a fixed function $h$ then $\alpha_{j\to i}$ represents the (relative) synaptic weight of neuron $j$ over neuron $i$. Notice that in the present paper we allow the functions $h_{j \rightarrow i}$ to be random and thus the synaptic weights to be random as well (as in \cite{faugeras2009constructive} for instance).

To model a transition in the behaviour of the network at the shifting time $t=0$, the distribution of $N\cap\mathbb{R}_-$ is considered as an initial condition of the dynamics of the point process and may be different from the distribution of a Hawkes process. Therefore, to specify the dependence of the dynamics (on the positive times) upon the initial condition, the following form of intensity can be considered:
\begin{equation}\label{eq:def:Hawkes:passe}
\lambda_{t}= \Phi\left( \int_{-\infty}^{t-} h(t-z) N(dz) \right) = \Phi\left( \int_{0}^{t-} h(t-z) N(dz) + F(t) \right),
\end{equation}
where $F(t):=\int_{-\infty}^{0} h(t-z) N(dz)$ models, in a Hawkes manner, the influence of the initial condition\footnote{Remark that the integral is performed over $(-\infty,0]$. In particular, it is possible that $0\in N$ with positive probability depending on the choice of the initial condition (namely the distribution of $N\cap\mathbb{R}_-$).}. This choice of $F$ is taken from \cite{chevallier2015microscopic}. However, other choices are conceivable. For example, more general functions $F$ may describe a stimulus at a given time $t_0<0$ which is more convenient for peristimulus analyses like \cite{Pouzat2009STAR}.

However, standard Hawkes processes fail to model, in a convenient way, the neurophysiological constraint known as \emph{refractory period}, that is the fact that a neuron cannot spike twice in a too short delay. This is the main reason why we allow the intensity of the Hawkes process to depend upon the age in the present study. In comparison with \eqref{eq:def:Hawkes:univarie}, one may represent strict refractory period by considering, for instance, the following form of intensity:
\begin{equation}\label{eq:def:Hawkes:refractory:period}
\lambda_{t}= \Phi\left( \int_{0}^{t-} h(t-z) N(dz) \right)\mathds{1}_{S_{t-}\geq \delta},
\end{equation}
where $(S_{t-})_{t\geq 0}$ is the predictable age process associated with $N$ and $\delta$ is a parameter corresponding to the time length of the strict refractory period of a neuron.
This sounds as an alternative to the strategy used in \cite{chevallier2015detection}. Therein, refractory periods are described by choosing, in the standard formulation of Hawkes processes, strongly negative self-interaction functions at a very short range. The strategy used in the present article is more flexible: synaptic integration and refractory period involve different aspects of the physiology of a neuron and so we prefer to describe each of them by different elements in the modelling.\\

Mean-field approaches have been already used to pass from a microscopic to a macroscopic description of neural networks. Taking for granted that the network is symmetric enough, the mean-field modelling sounds quite fair. Indeed neural networks admit a large number of vertices and are highly connected (see \cite{faugeras2009constructive} where the mean-field approximation of the columns structure in the visual cortex is discussed or \cite{bojak2010connecting} where mean-field models are related to experimental data recorded at a macroscopic scale). One may distinguish three types of models: intrinsically spike generating models (like the FitzHugh–Nagumo model \cite{luccon2014mean}), threshold spike generating models (like the integrate-and-fire model \cite{brunel1999fast,delarue2015global,delarue2015particle}) and point processes models (\cite{fournier2014toy} or \cite{galves2015modeling,hodara2014}).\\

As usual with McKean-Vlasov dynamics, the asymptotic evolution (when $n$ goes to infinity) of the distribution of the population at hand can be described as the solution of a nonlinear partial differential equation (PDE).
In the present article, the candidate to describe the dynamics at a macroscopic level is the following age structured system of nonlinear PDEs studied by Pakdaman, Perthame and Salort in a series of articles \cite{pakdaman2010dynamics,pakdaman2013relaxation,pakdaman2014adaptation}.
\begin{equation}\label{eq:PPS:intro} \tag{PPS}
\begin{cases}
\displaystyle \frac{\partial n\left(s,t\right)}{\partial t}+
\frac{\partial n\left(s,t\right)}{\partial s}+
p\left(s,X\left(t\right)\right)n\left(s,t\right)=0,\\
\displaystyle m\left(t\right):=n\left(0,t\right)=
\int_{0}^{+\infty}p\left(s,X\left(t\right)\right)n\left(s,t\right)ds.
\end{cases}
\end{equation}
Here, $n(s,t)$ represents the probability density of the age $s$ of a neuron at time $t$ where the age of a neuron is the delay since its last spike. Of course, the definition of the age of a neuron fits with the definition of the age associated with a point process as soon as the spike train is modelled by a point process.
The function $p$ represents the firing rate which may depend on the age $s$. As already explained, this dependence describes for instance the phenomenon of refractory period (e.g. $p(s,x)=\mathds{1}_{s\geq \delta}$ for some $\delta>0$). The function $p$ may also depend on the global activity of the network which is denoted by $X(t):=\int_0^t d(z) n(0,t-z) dz$ where $d$ is some delay function.
This global (deterministic) variable $X(t)$ corresponds to the mean of the integral that appears in \eqref{eq:def:Hawkes:univarie}. This correspondence forms the basis of the previous work \cite{chevallier2015microscopic} where a bridge is made between a modified version of \eqref{eq:PPS:intro} and the distribution of the age of a single neuron (modelled by a point process). From a neural network point of view, this distribution can of course be recovered as the limit of the empirical distribution associated with a network of i.i.d. neurons.

The study of the link between the \eqref{eq:PPS:intro} system and a mean-field interacting neural network (modelled by point processes) was left as an open question in \cite{chevallier2015microscopic}. The heuristic of this mean-field interpretation comes from the specific structure of the variable $X(t)$ which brings out a non-linearity of the McKean-Vlasov type. One of the main purpose of the present paper is to answer that left open question.
To be precise, this kind of study is performed in a preliminary work \cite{quininao2015microscopic} for a firing rate $p$ that is continuous and non-decreasing in both variables and under Markovian assumptions. Transposed to the Hawkes framework, this last point corresponds to interaction functions of the form $h_{j\to i}(t)=e^{-\beta(t-\tau_{j})}\mathds{1}_{[\tau_{j},+\infty)}(t)$ where $\beta$ is a constant and the $\tau_{j}$'s are i.i.d. random variables describing the propagation time of the signal from the neuron to the network.
The convergence of the empirical measure is discussed in \cite{quininao2015microscopic} when $p$ is continuous only but without any rate of convergence.
In the present study, rates of convergence are given for non Markovian Hawkes processes (that is non necessary exponential interaction functions) as well as for firing rates that are discontinuous with respect to the age, like \eqref{eq:def:Hawkes:refractory:period} for instance. However, we make the crucial assumption that the firing rate $p$ is Lipschitz continuous with respect to the second variable.\\

To sum up, we call Age Dependent Random Hawkes Process (ADRHP) a multivariate age dependent Hawkes process (like \eqref{eq:def:Hawkes:refractory:period} for instance) with some general dependence with respect to the initial condition~\eqref{eq:def:Hawkes:passe} and with some randomness regarding the interaction functions $h_{j\to i}$.
This article has two main purposes: extend the mean-field approximation obtained in \cite{delattre2016} to this generalization of Hawkes processes and establish a proper link between the microscopic modelling of individual neurons given by a $n$-particle system of mean-field interacting age-dependent Hawkes processes (like \eqref{eq:intensity:ADHP} for instance) and the macroscopic modelling given by the \eqref{eq:PPS:intro} system. 

The paper is organized as follows. In section \ref{sec:ADRHP}, we introduce ADRHPs and we show how to represent them as solutions of an SDE driven by a Poisson noise. 
As a by-product of this representation, we get, on the one hand, the existence of such processes, and on the other hand, an efficient way to get a coupling between our $n$-particle system and $n$ i.i.d. limit processes.
As a first step towards the mean-field approximation, the limit dynamics is studied in Section \ref{sec:study:limit:dynamics}. Existence and uniqueness of a solution of the \eqref{eq:PPS:intro} system, which is our candidate to drive the limit dynamics, are proved in Theorem \ref{thm:existence:uniqueness:PPS}. 
As a consequence, we get the existence of point processes of the McKean-Vlasov type whose intensity depends on both the time and the age.
In Section \ref{sec:mean:field}, these processes are proved to be the mean-field approximation of age dependent random Hawkes processes (Theorem \ref{thm:counting:process:convergence} and Corollary \ref{cor:mean:field:approx:counting:and:age:process}) using coupling arguments under either of the two following main assumptions: the intensity is bounded or the intensity does not depend on the age. 
Notice that even when the intensity does not depend on the age, the results presented here extend the ones given in \cite{delattre2016} since random interaction functions $h_{j\to i}$ as well as dependences with respect to the dynamics before time $0$ cannot be taken into account in \cite{delattre2016}.
Finally, the link between age dependent random Hawkes processes and the \eqref{eq:PPS:intro} system is given by Corollary \ref{cor:mean:field:approx:counting:and:age:process}. For sake of readability, most of the computations and technical lemmas are given in two appendices.

\paragraph{General notations}
\begin{itemize}
\item The space of continuous function from $E$ to $\mathbb{R}$ is denoted by $\mathcal{C}(E)$.
\item The space of Radon (resp. probability) measures on $E$ is denoted by $\mathcal{M}(E)$ (resp. $\mathcal{P}(E)$).
\item For $\nu$ in $\mathcal{P}(E)$, $X\sim \nu$ means that $X$ is a random variable distributed according to $\nu$.
\item For $f:E\rightarrow \mathbb{R}$, $|| f ||_{1}$, $|| f ||_{2}$ and $|| f ||_{\infty}$ respectively denote the $L^{1}$, $L^{2}$ and $L^{\infty}$ norms of $f$.
\end{itemize}

\section{Age dependent random Hawkes processes}
\label{sec:ADRHP}
In all the sequel, we focus on locally finite simple point processes, $N$, on $(\mathbb{R},\mathcal{B}(\mathbb{R}))$ that are random countable sets of points of $\mathbb{R}$ such that for any bounded measurable set $K\subset \mathbb{R}$, the number of points in $N\cap K$ is finite almost surely (a.s.). The associated points define an ordered sequence of points $(T_{n})_{n\in \mathbb{Z}}$.
For a measurable set $A$, $N(A)$ denotes the number of points of $N$ in $A$. We are interested in the behaviour of $N$ on $(0,+\infty)$ ($N\cap \mathbb{R}_{-}$ is regarded as an initial condition) and we denote $t\in\mathbb{R}_{+} \mapsto N_t:=N((0,t])$ the associated counting process. Furthermore, the point measure associated with $N$ is denoted by $N(dt)$. In particular, for any non-negative measurable function $f$, $\int_{\mathbb{R}} f(t) N(dt) = \sum_{i\in \mathbb{Z}} f(T_{i})$. For any point process $N$, we call \emph{age process} associated with $N$ the process $(S_{t})_{t\geq 0}$ given by
\begin{equation}\label{eq:def:age:process}
S_{t}= t-\sup\{T\in N,\, T\leq t\}, \quad \text{ for all } t\geq 0.
\end{equation}
In comparison with the age process, we call \emph{predictable age process} associated with $N$ the predictable process $(S_{t-})_{t\geq 0}$ given by
\begin{equation}\label{eq:def:age:process:predictable}
S_{t-}= t-\sup\{T\in N,\, T<t\}, \quad \text{ for all } t>0,
\end{equation}
and extended by continuity in $t=0$.

We work on a filtered probability space $(\Omega,\mathcal{F},(\mathcal{F}_{t})_{t\geq 0},\mathbb{P})$ and suppose that the canonical filtration associated with $N$, namely $(\mathcal{F}_{t}^{N})_{t\geq 0}$ defined by $\mathcal{F}_{t}^{N}:=\sigma(N\cap (-\infty,t))$, is such that for all $t\geq 0$, $\mathcal{F}_{t}^{N}\subset	\mathcal{F}_{t}$. 
Let us denote $\mathbb{F}:=(\mathcal{F}_{t})_{t\geq 0}$.
We call $\mathbb{F}$-(predictable) intensity of $N$ any non-negative $\mathbb{F}$-predictable process $(\lambda_{t})_{t\geq 0}$ such that $(N_{t}-\int_{0}^{t} \lambda_{s}ds)_{t\geq 0}$ is an $\mathbb{F}$-local martingale. Informally, $\lambda_{t}dt$ represents the probability that the process $N$ has a new point in $[t,t+dt]$ given $\mathcal{F}_{t-}$. Under some assumptions that are supposed here, this intensity process exists, is essentially unique and characterizes the point process (see \cite{Bremaud_PP} for more insight).. In particular, since $N$ admits an intensity, for any $t\geq 0$, the probability that $t$ belongs to $N$ is null. Moreover, notice the following properties satisfied by the age processes:
\begin{itemize}
\item the two age processes are equal for all $t\geq 0$ except the positive times $T$ in $N$ (almost surely a set of null measure in $\mathbb{R}_{+}$),
\item for any $t\geq 0$, $S_{t-}=S_{t}$ almost surely (since $N$ admits an intensity),
\item and the value $S_{0-}=S_{0}$ is entirely determined by $N\cap\mathbb{R}_-$ and is well-defined as soon as there is a point therein.
\end{itemize}
\medskip

In analogy with the study of the dynamics of a variable over time, we use a dichotomy between the behaviour of the point process before time $0$ (which is treated as an initial condition) and its behaviour after time $0$ (which is supposed to admit a ``Hawkes type'' intensity). For every point process $N$, we denote $N_{-}=N\cap\mathbb{R}_-$ and $N_{+}=N\cap (0,+\infty)$. In the rest of the paper, a point process on $\mathbb{R}$ is characterized by:
\begin{enumerate}
\item the distribution of $N_{-}$, namely $\zeta_{N_{-}}$, which gives the dynamics of $N$ on $\mathbb{R}_{-}$;
\item the $\mathbb{F}$-predictable intensity $\lambda_{t}$, which gives the dynamics of $N$ on $(0,+\infty)$.
\end{enumerate}
In particular, $\zeta_{N_{-}}$ characterizes the distribution of $T_{0}$ that is the last point (spike) before time $0$.
Notice that the $\sigma$-algebra $\mathcal{F}_{0}$ is such that $N_{-}$ is $\mathcal{F}_{0}$-measurable.\\

\subsection{Parameters of the model}

The definition of an \emph{age dependent random Hawkes process} (ADRHP) is given bellow, but let us first introduce the parameters of the model:

$\bullet$ a positive integer $n$ which is the number of particles (e.g. neurons) in the network (for $i=1,\dots ,n$, $N^{i}$ represents the occurrences of the events (e.g. spikes) associated with the particle $i$);

 $\bullet$ a distribution $\zeta_{N_{-}}$ determining the initial conditions $(N_{-}^{i})_{i=1,..,n}$ which are i.i.d. point processes on $\mathbb{R}_{-}$ distributed according to $\zeta_{N_{-}}$;

 $\bullet$ a distribution $\mu_{H}$ determining the matrix of interaction functions $\mathbf{H}=(H_{ij})_{1\leq i,j\leq n}$ where $H_{ij}:\mathbb{R}_{+}\to \mathbb{R}$ are $\mathcal{F}_{0}$-measurable random functions distributed according to $\mu_{H}$ such that
\begin{equation}\label{eq:indep:Hij}
{\small
\begin{cases}
\text{for any fixed $i=1,\dots ,n$, the variables $H_{i1},\dots ,H_{in}$ are independent,}\\
\text{the vectors $(H_{i1},\dots ,H_{in})$ are exchangeable (with respect to $i$),}\\
\text{the matrix {\bf H} is independent from the initial conditions $(N_{-}^{i})_{i=1,..,n}$;}
\end{cases}
}
\end{equation}

 $\bullet$ a distribution $\nu_{F}$ determining the matrix of functions $\mathbf{F}=(F_{ij})_{1\leq i,j\leq n}$ where $F_{ij}:\mathbb{R}_{+}\to \mathbb{R}$ are $\mathcal{F}_{0}$-measurable random functions distributed according to $\nu_{F}$ such that
\begin{equation}\label{eq:indep:Fij}
{\small
\begin{cases}
\text{for any fixed $i=1,\dots ,n$, the variables $F_{i1},\dots ,F_{in}$ are independent,}\\
\text{the vectors $(F_{i1},\dots ,F_{in})$ are exchangeable (with respect to $i$);}
\end{cases}
}
\end{equation}

 $\bullet$ an intensity function $\Psi:\mathbb{R}_{+}\times \mathbb{R} \rightarrow \mathbb{R}_{+}$.

Note that the functions $H_{ij}$'s can in particular be equal to a given deterministic function $h$ which corresponds to more standard Hawkes processes.

\begin{rem}
\textbf{\emph{On the exchangeability.}} Assumptions \eqref{eq:indep:Hij} and \eqref{eq:indep:Fij} mean that a single particle receives i.i.d. interactions from its neighbours and that the particles are exchangeable: one can permute the particles without modifying their joint distribution.
\medskip

\noindent \textbf{\emph{On the synaptic weights.}} The link between synaptic weights (that is the strength with which one given neuron influences an other one)  and interaction functions can be well emphasized by the following choice of interaction functions.
Consider a fixed function $h:\mathbb{R}_{+}\to \mathbb{R}$ and, independently of everything else, a sequence $(\alpha_{j})_{j=1,\dots ,n}$ of i.i.d. random variables with values in $[0,1]$. Then, $(H_{ij})_{1\leq i,j\leq n}$ defined by $H_{ij}=\alpha_{j}h$ satisfies \eqref{eq:indep:Hij}. The $\alpha_{j}$'s represent the (relative) synaptic weight of neuron $j$ over all the other ones.

The interaction functions, even if they are random, are fixed at time $0$. The dynamics of synaptic weights is not taken into account here.
\medskip

\noindent \textbf{\emph{On the initial condition.}} Unlike the matrix ${\bf H}$, the matrix ${\bf F}$ can depend on the initial conditions as it can be seen in the following particular case which is derived from \eqref{eq:def:Hawkes:passe} for instance. For the same matrix ${\bf H}=(H_{ij})_{1\leq i,j\leq n}$ as in \eqref{eq:indep:Hij}, we may choose for all $1\leq i,j\leq n$, the function $F_{ij}:\mathbb{R}_{+}\rightarrow \mathbb{R}$ defined for all $t\geq0$ by
\begin{equation}\label{eq:def:F0}
F_{ij}(t)= \int_{-\infty}^{0} H_{ij}(t-z) N^{j}_{-}(dz).
\end{equation}
These random functions are $\mathcal{F}_{0}$-measurable and they satisfy the first two lines of \eqref{eq:indep:Fij} thanks to the independence of the $H_{ij}$'s and the $N^{j}_{-}$'s. Hence, one can consider the intensity given by \eqref{eq:def:ADRHP:intensity} with such a choice of $F$ to represent the contribution of the processes $(N^{i}_{-})_{i=1,\dots ,n}$ to the dynamics after time $0$. In this example, the $F_{ij}$'s are obviously dependent from the $N^{j}_{-}$'s.
\medskip

\noindent \textbf{\emph{On the post-stimulus study.}} In the case of neurons modelling, one can model external inputs via the functions $F_{ij}$. For example, one could take $F_{ij}=H_{ij}(t-\tau)$ where $\tau$ is some non-positive real number that may be random (independent of anything else) modelling that all the neurons have spiked at the same time $\tau<0$ thanks to a common stimulus.
\end{rem}

\subsection{Definition via the intensity}

The definition of an age dependent random Hawkes process is given by providing the form of its intensity.

\begin{defn}\label{def:Hawkes:intensity}
\begin{sloppypar}
An age dependent random Hawkes process (ADRHP) with parameters $(n,\mu_{H},\nu_{F},\Psi,\zeta_{N_{-}})$ is a family $(N^{i})_{i=1,..,n}$ of point processes on $\mathbb{R}$ such that $(N_{-}^{i})_{i=1,..,n}$ is a family of i.i.d. point processes on $\mathbb{R}_{-}$ distributed according to $\zeta_{N_{-}}$ and $(N_{+}^{i})_{i=1,..,n}$ is a family of point processes on $\mathbb{R}_{+}$ with $\mathbb{F}$-intensity given for all $i=1,\dots ,n$ by
\end{sloppypar}
\begin{equation}\label{eq:def:ADRHP:intensity}
\lambda^{i}_{t}=\Psi\left(S^{i}_{t-}, \frac{1}{n} \sum_{j=1}^{n} \left( \int_{0}^{t-} H_{ij}(t-z) N_{+}^{j}(dz) + F_{ij}(t) \right) \right) ,
\end{equation}
\begin{sloppypar}
where $(S^{i}_{t-})_{t\geq 0}$ is the predictable age process associated with $N^{i}$ defined in \eqref{eq:def:age:process:predictable} and $(H_{ij})_{1\leq i,j\leq n}$ (respectively $(F_{ij})_{1\leq i,j\leq n}$) is a random matrix with entries distributed according to $\mu_{H}$ (resp. $\nu_{F}$) and satisfying \eqref{eq:indep:Hij} (resp. \eqref{eq:indep:Fij}).
\end{sloppypar}
\end{defn}

Notice that the intensities depend on the predictable age processes and not the standard ones since an intensity process must be predictable. Furthermore, since the auto-interaction given by $H_{ii}$ is scaled by $1/n$, it vanishes when $n$ goes to infinity and so the asymptotic behaviour  proved in this article (Corollary \ref{cor:mean:field:approx:counting:and:age:process}) remains the same if one assumes that $H_{ii}=0$.

An age dependent random Hawkes process admits two different behaviours:
\begin{enumerate}
\item before time $0$, the processes $(N^{i}_{-})_{i=1,\dots ,n}$ are independent and identically distributed;
\item after time $0$, the processes $(N^{i}_{+})_{i=1,\dots ,n}$ are dependent (in general) and driven by their respective intensities which can be different from one process to another. 
\end{enumerate}

The dichotomy of behaviours can model a change of regime at time $t=0$. It should be interesting to see whether the results could be extended to initial conditions given by a mean-field dynamics and not necessarily i.i.d. ones. However, it is not in the scope of this article. 

\begin{rem}
\textbf{\emph{On the randomness.}} Given $\mathcal{F}_{0}$, the  randomness of $\lambda^{i}_{t}$ in Equation~\eqref{eq:def:ADRHP:intensity} only lies in the point measures $N_{+}^{j}(dz)$ and the predictable age process $(S^{i}_{t-})_{t\geq 0}$. These intensities, and so the point processes, are not exchangeable given $\mathcal{F}_{0}$. However, they are exchangeable when they are considered with respect to all the randomness (including the $N^{i}_{-}$'s, $H_{ij}$'s and $F_{ij}$'s).
\medskip

\noindent \textbf{\emph{Particular case for $\Psi$.}} As presented in the introduction (see Equation \eqref{eq:def:Hawkes:refractory:period}), a particular case we have in mind in this study is when there exists a function $\Phi:\mathbb{R}\to \mathbb{R}_{+}$ and a non-negative real number $\delta$ such that 
\begin{equation}\label{eq:example:HPRP}
\Psi(s,x)=\Phi(x)\mathds{1}_{s\geq \delta}.
\end{equation}
This particular choice of $\Psi$ provides an interesting modelling of the strict refractory period of a neuron. Furthermore, when $\delta=0$, there is no refractory period and one recovers more standard Hawkes processes. In particular, if $\mu_{H}$ is the Dirac mass located at some fixed function $h$ and  $\nu_{F}$ is the Dirac mass located at the null function, then one recovers the Hawkes processes studied in \cite{delattre2016}. Remark that the exchangeability of the Hawkes processes studied in \cite{delattre2016} is obvious since they have the same intensity at each time $t$.
\end{rem}

\subsection{List of assumptions}

In the present article, several assumptions on the parameters of the model are used depending on the context. For sake of simplicity, all these assumptions are gathered here.
\paragraph{Main assumptions}\ 
\smallskip

\newlength{\asslabel}
\setlength{\asslabel}{1.6cm}
\newlength{\assdef}
\setlength{\assdef}{11.2cm}
\newcommand{\spacenot}{\vspace{-0.2cm}}
\noindent
\begin{tabular}{>{\centering}p{\asslabel}|p{\assdef}}
\phantomsection
\label{ass:initial:condition:density} 
\spacenot $\left(\mathcal{A}^{\zeta_{N_{-}}}_{u^{\rm in}}\right)$:   & \textbf{Age at time $0$ admits a bounded probability density.}

If $N_{-}$ is distributed according to $\zeta_{N_{-}}$ ($N_{-}\sim \zeta_{N_{-}}$) and $T_{0}$ denotes the closest point of $N_{-}$ to 0, then $-T_{0}$ admits a density with respect to the Lebesgue measure denoted by $u^{\rm in}$ (``in'' stands for ``initial''). Furthermore, $u^{\rm in}$ is uniformly bounded.
\end{tabular}

\bigskip
\noindent
\begin{tabular}{>{\centering}p{\asslabel}|p{\assdef}}
\phantomsection
\label{ass:mu:H:infty} 
\spacenot $\left(\mathcal{A}^{\mu_{H}}_{\infty}\right)$:      & 	\textbf{Interaction functions: local integrability.}

If $H\sim \mu_{H}$, then there exists a deterministic function $G:\mathbb{R}_{+}\to \mathbb{R}$ such that a.s., for all $t\geq 0$, $|H(t)|\leq G(t)$.
The smallest possible deterministic function $G$, denoted by $M_{\mu_{H}}$, is moreover supposed to be locally integrable. In particular, $\mathbb{E}\left[ H(t) \right]$ is well-defined and we let $m_{\mu_{H}}(t):=\mathbb{E}\left[ H(t) \right]$.
\end{tabular}

\bigskip
\noindent
\begin{tabular}{>{\centering}p{\asslabel}|p{\assdef}}
\phantomsection
\label{ass:nu:F:locally:bounded}
\spacenot $\left(\mathcal{A}^{\nu_{F}}_{1}\right)$:      & \textbf{Expectation of the functions $F_{ij}$.}

If $F\sim\nu_{F}$, then $t\in\mathbb{R}_{+}\mapsto \mathbb{E}\left[ |F(t)| \right]$ is locally bounded. In particular, for all $t\geq 0$,  $\mathbb{E}\left[ F(t) \right]$ is well-defined and we let $m_{\nu_{F}}(t):=\mathbb{E}\left[ F(t) \right]$.
\end{tabular}

\bigskip
\noindent
\begin{tabular}{>{\centering}p{\asslabel}|p{\assdef}}
\phantomsection
\label{ass:Psi:Lipschitz}
\spacenot $\left(\mathcal{A}^{\Psi}_{\rm Lip}\right)$:      & \textbf{Lipschitz continuity.}

The function $\Psi:\mathbb{R}_{+}\times \mathbb{R}\to\mathbb{R}_{+}$ is uniformly Lipschitz continuous with respect to the second coordinate:
there exists a constant $C>0$ such that, for all $s\geq 0$, the function $x\mapsto \Psi(s,x)$ is Lipschitz with constant $C$. The smallest constant $C$ is denoted by ${\rm Lip}(\Psi)$.
 Furthermore, $s\in\mathbb{R}_{+} \mapsto \Psi(s,0)$ is uniformly bounded.
\end{tabular}

\smallskip
\noindent
\begin{tabular}{>{\centering}p{\asslabel}|p{\assdef}}
\phantomsection
\label{ass:Psi:uniformly:bounded}
\spacenot $\left(\mathcal{A}^{\Psi}_{\infty}\right)$:      & \textbf{Uniformly bounded.}

The function $\Psi$ is uniformly bounded, that is $|| \Psi ||_{\infty}<+\infty$.
\end{tabular}

\smallskip
\noindent
\begin{tabular}{>{\centering}p{\asslabel}|p{\assdef}}
\phantomsection
\label{ass:Psi:=Psi0}
\spacenot $\left(\mathcal{A}_{\Psi=\Psi_{0}}\right)$:      & \textbf{The intensity does not depend on the age.}

There exists a function $\Psi_{0}:\mathbb{R}\to \mathbb{R}_{+}$ such that, for all $s\geq 0$, $\Psi(s,\cdot)=\Psi_{0}(\cdot)$. In this case, if (\hyperref[ass:Psi:Lipschitz]{$\mathcal{A}^{\Psi}_{\rm Lip}$}) is satisfied then ${\rm Lip}(\Psi)$ is rather denoted by ${\rm Lip}(\Psi_{0})$.
\end{tabular}

\paragraph{Additional assumptions}\ 
\smallskip

\noindent
\begin{tabular}{>{\centering}p{\asslabel}|p{\assdef}}
\phantomsection
\label{ass:initial:condition:bounded}
\spacenot $\left(\mathcal{A}^{\zeta_{N_{-}}}_{\infty}\right)$:     & \textbf{The age at time $0$ is bounded.}

If $N_{-}\sim \zeta_{N_{-}}$ and $T_{0}$ denotes the closest point of $N_{-}$ to 0, then $-T_{0}$ is upper bounded a.s. that is there exists a constant $C>0$ such that $-T_{0}\leq C$ a.s. The smallest possible constant $C$ is denoted by $M_{T_{0}}$.
\end{tabular}

\bigskip
\noindent
\begin{tabular}{>{\centering}p{\asslabel}|p{\assdef}}
\phantomsection
\label{ass:mu:H:locally:square:integrable}
\spacenot $\left(\mathcal{A}^{\mu_{H}}_{\infty,2}\right)$:    & \textbf{Interaction functions: square local integrability.}

(\hyperref[ass:mu:H:infty]{$\mathcal{A}^{\mu_{H}}_{\infty}$}) is satisfied and $M_{\mu_{H}}$ is furthermore locally square integrable.
\end{tabular}

\bigskip
\noindent
\begin{tabular}{>{\centering}p{\asslabel}|p{\assdef}}
\phantomsection
\label{ass:nu:F:variance}
\spacenot $\left(\mathcal{A}^{\nu_{F}}_{2}\right)$:      & \textbf{Variance of the functions $F_{ij}$.}

(\hyperref[ass:nu:F:locally:bounded]{$\mathcal{A}^{\nu_{F}}_{1}$}) is satisfied and if $F\sim\nu_{F}$, then for all $t\geq 0$, $F(t)$ admits a variance denoted by $V_{\nu_{F}}(t)$ satisfying that for all $t\geq 0$, $\int_{0}^{t} V_{\nu_{F}}(t')^{1/2} dt'<+\infty$. Furthermore, $m_{\nu_{F}}$ is a continuous function.
\end{tabular}

\medskip

\noindent Notice that:
\begin{itemize}
\item Assumptions (\hyperref[ass:mu:H:infty]{$\mathcal{A}^{\mu_{H}}_{\infty}$}), (\hyperref[ass:nu:F:locally:bounded]{$\mathcal{A}^{\nu_{F}}_{1}$}) and (\hyperref[ass:Psi:Lipschitz]{$\mathcal{A}^{\Psi}_{\rm Lip}$}) are used to prove the existence of the $n$-particle system.
\item Assumptions (\hyperref[ass:mu:H:locally:square:integrable]{$\mathcal{A}^{\mu_{H}}_{\infty,2}$}), (\hyperref[ass:nu:F:variance]{$\mathcal{A}^{\nu_{F}}_{2}$}) and (\hyperref[ass:Psi:Lipschitz]{$\mathcal{A}^{\Psi}_{\rm Lip}$}) are used to prove the mean-field approximation under the addition of either (\hyperref[ass:Psi:uniformly:bounded]{$\mathcal{A}^{\Psi}_{\infty}$}) and (\hyperref[ass:initial:condition:density]{$\mathcal{A}^{\zeta_{N_{-}}}_{u^{\rm in}}$})
or (\hyperref[ass:Psi:=Psi0]{$\mathcal{A}_{\Psi=\Psi_{0}}$}).
\end{itemize}

\begin{rem}
There are two strong assumptions regarding the intensity function $\Psi$: Lipschitz continuity and boundedness. The continuity assumption is rather standard to prove existence of non-explosion of Hawkes processes and the mean-field approximation by the coupling method. The boundedness assumption is (up to our knowledge) necessary to prove the results stated in Section \ref{sec:study:limit:dynamics} (granting the well-posedness of the limit process). Nevertheless, once this well-posedness is given, we could relax the uniform boundedness assumption to a local boundedness one, provided that some a priori bounds hold for the two variables of $\Psi$: the limit age $\overline{S}_{t-}$ and the limit interaction variable $\int_{0}^{t} h(t-z) \overline{\lambda}(z)dz + f_{0}(t)$ -- see Equation \eqref{eq:mckean:vlasov:intensity}. The interested reader is referred to \cite{quininao2015microscopic} where this kind of argument is used.
\end{rem}

\subsection{Representation via a stochastic differential equation}

Definition \ref{def:Hawkes:intensity} describes ADRHPs as weak solutions. It characterizes their distribution but not their path-wise dynamics. As it is well emphasized in \cite{massoulie_1998}, point processes can be either represented as weak solutions thanks to their stochastic intensity or represented as strong solutions of a stochastic differential equation (SDE) driven by Poisson noise. The idea to represent point processes as \emph{strong} solutions of SDEs driven by Poisson measures was first introduced by Lewis and Shedler \cite{Lewis_Simul}, for inhomogeneous Poisson process, and extended by Ogata \cite{ogata1981on} (thinning procedure) under some weak assumptions on the intensity. 
It says that, if $N$ admits $(\lambda_{t})_{t\geq 0}$ as a $\mathbb{F}$-predictable intensity, then the point measure associated with $N$ can be represented by $N(dt)=\Pi(dt\times [0,\lambda_{t}])$ where $\Pi$ is a Poisson measure with intensity $1$ on $\mathbb{R}_{+}^{2}$.
This has been used to show existence or stability results for some classes of point processes by Br\'emaud and Massouli\'e in \cite{Bre_Massou} or \cite{massoulie_1998} and more recently to exhibit some suitable coupling between interacting Hawkes processes and their mean-field approximation in \cite{delattre2016}.
We introduce here the representation of ADRHPs based on such a thinning procedure.

\begin{representation}\label{def:Hawkes:Thinning}
Let $(N_{-}^{i})_{i\geq 1}$ be some i.i.d. point processes on $\mathbb{R}_{-}$ distributed according to $\zeta_{N_{-}}$.
Let $(H_{ij})_{1\leq i,j\leq n}$ (respectively $(F_{ij})_{1\leq i,j\leq n}$) be a random matrix with entries distributed according to $\mu_{H}$ (resp. $\nu_{F}$) and satisfying \eqref{eq:indep:Hij} (resp. \eqref{eq:indep:Fij}).
Let $(\Pi^{i}(dt,dx))_{i\geq 1}$ be some i.i.d. $\mathbb{F}$-Poisson measures with intensity $1$ on $\mathbb{R}_{+}^{2}$.

Let $(N_{t}^i)^{i=1,..,n}_{t \geq 0}$ be a family of counting processes such that, for $i=1,..,n$ and all $t\geq 0$,
\begin{equation} \label{eq:Hawkes:Thinning}
N_{t}^i= \int_0^t \int_0^\infty \mathds{1}_{ \left\{ x\leq \Psi\left( S^{i}_{t'-}, \frac{1}{n} \sum_{j=1}^{n} ( \int_{0}^{t'-} H_{ij}(t'-z) N_{+}^{j}(dz) + F_{ij}(t') ) \right) \right\}} \, \Pi^i(dt',dx),
\end{equation}
where $(S^{i}_{t-})_{t\geq 0}$ is the predictable age process associated with $N^{i}=N_{-}^{i}\cup N_{+}^{i}$ and $N_{+}^{i}$ is the point process associated with the counting process $(N_{t}^i)_{ t \geq 0}$. Then, $(N^{i})_{i=1,..,n}$ is an age dependent random Hawkes process with parameters $(n,\mu_{H},\nu_{F},\Psi,\zeta_{N_{-}})$.
\end{representation}

This representation is mainly used in this paper in order to provide a suitable coupling between ADRHPs and i.i.d. point processes describing the mean-field dynamics.

Going back and forth between the weak solution of Definition \ref{def:Hawkes:intensity} and the strong solution of Representation \ref{def:Hawkes:Thinning}  is classic: the thinning Theorem (see \cite[Lemma 2]{Bre_Massou} or \cite[Theorem B.11]{chevallier2015microscopic} for a complete proof) states that a strong solution is also a weak solution; and the Poisson inversion \cite[Lemma 4]{Bre_Massou} states that, from a weak solution $(N^{i})_{i=1,\dots ,n}$, one can construct Poisson measures on an enlarged probability space such that~\eqref{eq:Hawkes:Thinning} is satisfied.

At this stage, one has two equivalent concepts of ADRHPs but no result on the existence of such processes. Indeed, if there is too much self-excitation, then there may be an infinite number of points in finite time.
In the present paper, point processes that do not explode in finite time are considered and, thanks to Representation \ref{def:Hawkes:Thinning}, one can prove existence of these non-explosive processes.

\begin{prop}\label{prop:Thinning:well:posed}
\begin{sloppypar}
Under  (\hyperref[ass:mu:H:infty]{$\mathcal{A}^{\mu_{H}}_{\infty}$}), (\hyperref[ass:nu:F:locally:bounded]{$\mathcal{A}^{\nu_{F}}_{1}$}) and (\hyperref[ass:Psi:Lipschitz]{$\mathcal{A}^{\Psi}_{\rm Lip}$}), there exists an ADRHP  $(N^i)_{i=1,..,n}$ with parameters $(n,\mu_{H},\nu_{F},\Psi,\zeta_{N_{-}})$ such that $t\mapsto \mathbb{E}\left[ N_{t}^{1} \right]$ is locally bounded.
\end{sloppypar}
\end{prop}
This result can be challenging in an infinite dimensional framework like in \cite{delattre2016}. However, in our finite dimensional framework, it is quite clear since the ADRHP can be stochastically dominated by some multivariate \emph{linear} Hawkes process (thanks to the Lipschitz assumption (\hyperref[ass:Psi:Lipschitz]{$\mathcal{A}^{\Psi}_{\rm Lip}$})). Yet the well-posedness of linear Hawkes processes is standard thanks to their branching structure \cite{hawkes_1974}.
Nevertheless, a proof of Proposition \ref{prop:Thinning:well:posed} is given in  \ref{sec:proof:prop:Thinning:well:posed}.\\

\section{Study of the limit dynamics}
\label{sec:study:limit:dynamics}

The interactions between the point processes involved in the definition of an ADRHP are of mean-field type. Therefore, the limit version of Equation~\eqref{eq:Hawkes:Thinning} is proposed below in Equation \eqref{eq:limit:equation:counting:process} (informally, the empirical means involved in~\eqref{eq:Hawkes:Thinning} are replaced by their expected values). The \emph{limit equation} with parameters $(h,f_{0},\Psi,\zeta_{N_{-}})$ is given by
\begin{equation}\label{eq:limit:equation:counting:process}
\forall t>0,\  \overline{N}_t= \int_0^t \int_0^\infty \mathds{1}_{ \left\{ x\leq \Psi\left( \overline{S}_{t'-},  \int_{0}^{t'-} h(t'-z) \mathbb{E} \left[\overline{N}_{+}(dz) \right] + f_{0}(t') \right)  \right\}}  \, \Pi(dt',dx),
\end{equation}
where $h$ and $f_{0}$ are some functions from $\mathbb{R}_{+}$ to $\mathbb{R}$, $\Pi(dt', dx)$ is an $\mathbb{F}$-Poisson measure on $\mathbb{R}_{+}^{2}$ with intensity $1$ and $(\overline{S}_{t-})_{t\geq 0}$ is the predictable age process associated with $\overline{N}=\overline{N}_{-}\cup \overline{N}_{+}$ where $\overline{N}_{-}$ is a point process distributed according to $\zeta_{N_{-}}$ and $\overline{N}_{+}$ is the point process associated with the counting process $(\overline{N}_t)_{ t \geq 0}$. 

Looking simultaneously at Equations \eqref{eq:Hawkes:Thinning} and \eqref{eq:limit:equation:counting:process} shows that the empirical mean of the random interaction functions $H_{ij}$ (\emph{respectively the random functions $F_{ij}$}) are replaced by $h$ (\emph{resp. $f_{0}$}) which should be the mean interaction function $m_{\mu_{H}}$ (\emph{resp. $m_{\nu_{F}}$}). Moreover, the empirical mean of the point measures $N_{+}^{j}(dz)$ in \eqref{eq:Hawkes:Thinning} is replaced by the expectation of the point measure $\overline{N}_{+}(dz)$.

Finally, let us note that the dependence with respect to the predictable age process is still present in the limit equation. This matches with experimental data in neuroscience where refractory periods are highlighted \cite{berry1998refractoriness,fuortes1962interpretation,gerstner2002spiking}. By comparison, there is no such dependence in the limit process given in \cite{delattre2016} which is an inhomogeneous Poisson process.\\

This limit equation is used in the next section to provide suitable couplings to prove the mean-field approximation. 
Hence, the main point of this section is to prove the well-posedness of the limit equation \eqref{eq:limit:equation:counting:process}. However, to study the probabilistic formulation of the mean-field dynamics described in Equation \eqref{eq:limit:equation:counting:process}, one first needs to find a representation of the distribution of a possible solution of \eqref{eq:limit:equation:counting:process}. As a first step, we prove existence/uniqueness results for a linearisation of the \eqref{eq:PPS:intro} system (Proposition \ref{prop:existence:uniqueness:PPS:linear}) as well as give a representation of the solution given by the method of characteristics (Proposition \ref{prop:characteristics:PPS:linear}). The second step is to deduce existence/uniqueness results for the \eqref{eq:PPS:intro} system (Theorem \ref{thm:existence:uniqueness:PPS}) from the linearised system via a fixed point argument.
Then, the well-posedness of the limit equation \eqref{eq:limit:equation:counting:process} is proved thanks to the results obtained for the \eqref{eq:PPS:intro} system. Finally, the link between the \eqref{eq:PPS:intro} system and the processes defined by the limit equation is fully investigated.

Note that the analytical study of the fixed point equation satisfied by the expectation of the solution of \eqref{eq:limit:equation:counting:process} (as it is done in \cite{delattre2016}) can be extended to the case when the intensity does depend on the age. However, the results for the \eqref{eq:PPS:intro} system are valid in a more general framework so they are favoured here.

\subsection{Study of the linear system}
\label{sec:linear:system}

In comparison with the \eqref{eq:PPS:intro} system, the linear system studied below corresponds to the case where the firing rate $p$ in \eqref{eq:PPS:intro} is a function of the time $t$ and the age $s$ only. More precisely, we consider the system
\begin{equation}\label{eq:edp:PPS:linear}
\begin{cases} 
\displaystyle \frac{\partial u\left(t,s\right)}{\partial t}+\frac{\partial u\left(t,s\right)}{\partial s} + f(t,s) u\left(t,s\right)=0, \\
\displaystyle u\left(t,0\right)=\int_{s\in \mathbb{R}_{+}} f(t,s) u\left(t,s\right)ds ,
\end{cases}
\end{equation}
where $f$ is a bounded function.

We state uniqueness of the solution of this system in a measure space as a consequence of the uniqueness result stated in \cite{canizo2013measure}.
More precisely, the result is stated in $\mathcal{BC}(\mathbb{R}_{+},\mathcal{M}(\mathbb{R}))$ that is the space of bounded continuous curves on $\mathcal{M}(\mathbb{R})$ (the space of Radon measures on $\mathbb{R}$) endowed with the \emph{bounded Lipschitz norm} as considered in \cite{canizo2013measure}.
As we are interested in probability measures, let us remark that the bounded Lipschitz norm on $\mathcal{P}(\mathbb{R}_{+})$ is equivalent, thanks to the duality of Kantorovich-Rubinstein, to the modified $1$-Wasserstein distance defined by
\begin{equation}\label{eq:def:modified:Wasserstein:distance}
\tilde{W}_{1} (\mu, \nu):=\inf \mathbb{E}\left[ \min(|X-Y|,1) \right],
\end{equation}
where the infimum is taken over all joint distributions of the random variables $X$ and $Y$ with marginals $\mu$ and $\nu$ respectively.

Since measure solutions are considered, a weak form of the system is given. The following set of test functions is used:

\smallskip
\noindent$\mathcal{C}_{c,b}^{\infty}(\mathbb{R}_{+}^{2})\,
 \textrm{\begin{tabular}{|l} The function $\varphi$ belongs to $\mathcal{C}_{c,b}^{\infty}(\mathbb{R}_{+}^{2})$  if  \\ 
 $\ \bullet$ $\varphi$ is continuous, uniformly bounded, \\ 
 $\ \bullet$ $\varphi$ has  uniformly bounded derivatives of every order,\\ 
 $\ \bullet$ there exists $T>0$ such that $\varphi(t,s)=0$ for all $t>T$ and $s\geq 0$.
 \end{tabular}}$\\
 
The result stated below is a consequence of \cite[Theorem 2.4.]{canizo2013measure} in the same essence than the one presented in \cite[Section 3.3.]{canizo2013measure}.
Its proof is given in  \ref{sec:proof:prop:existence:uniqueness:PPS:linear}.

\begin{prop}\label{prop:existence:uniqueness:PPS:linear}
Assume that $f:\mathbb{R}_{+}\times \mathbb{R}_{+}\to \mathbb{R}$ is bounded and continuous (uniformly in the second variable) with respect to the first variable. Assume that $u^{\rm in}$ belongs to $\mathcal{M}(\mathbb{R}_{+})$.

Then, there exists a unique solution in the weak sense $u$ such that $t\mapsto u(t,\cdot)$ belongs to $\mathcal{BC}(\mathbb{R}_{+},\mathcal{M}(\mathbb{R}_{+}))$ of the system \eqref{eq:edp:PPS:linear} with initial condition $u(0,\cdot)=u^{\rm in}$. The weak sense means here that for every $\varphi$ in $\mathcal{C}^{\infty}_{c,b}(\mathbb{R}_{+}^{2})$,
\begin{multline}\label{eq:edp:PPS:linear:weak:sense}
\int_{\mathbb{R}_{+}^{2}} \left(\frac{\partial}{\partial t}+\frac{\partial}{\partial s}\right) \varphi\left(t,s\right)u\left( t, ds\right) dt +\int_{\mathbb{R}_{+}} \varphi(0,s) u^{\rm in}(ds)  \\
+ \int_{\mathbb{R}_{+}^{2}} [\varphi(t,0)-\varphi(t,s) ] f(t,s) u( t, ds) dt=0.
\end{multline}
\end{prop}

Remark that the system is mass-conservative (e.g. take a sequence of functions converging to $t\mapsto \mathds{1}_{[0,T]}(t)$ as test functions in the weak equation \eqref{eq:edp:PPS:linear:weak:sense}). 
As we are interested in probability measures as solutions, let us remark that the mass-conservation alone cannot ensure that the solution is a probability even if the initial condition is a probability. However, when the initial condition is a probability which admits a density, the method of characteristics shows that the solution of \eqref{eq:edp:PPS:linear} is a probability density function for all time $t\geq 0$.
\begin{prop}\label{prop:characteristics:PPS:linear}
Under the assumptions of Proposition \ref{prop:existence:uniqueness:PPS:linear}, assume that $u^{\rm in}$ is a probability which admits a density (denoted by $u^{\rm in}$ as well) with respect to the Lebesgue measure. Then, there exists a unique locally bounded function $u_{0}:\mathbb{R}_{+}\to \mathbb{R}$ (which is furthermore non-negative) such that $u$ defined by
\begin{numcases}{}
u(t,s) = u^{\rm in}(s-t) \exp\left( -\int_{0}^{t} f(t',s-t+t') dt'\right),  \quad \text{for } s\geq t \label{eq:characteristics:s>t}\\
u(t,s) = u_{0}(t-s) \exp\left( -\int_{0}^{s} f(t-s+s',s') ds'\right),  \quad \text{ for } t\geq s \label{eq:characteristics:t>s}
\end{numcases}
is the unique solution of \eqref{eq:edp:PPS:linear}.
In particular, 
\begin{itemize}
\item $u$ satisfies the second equation of \eqref{eq:edp:PPS:linear} in a strong sense,
\item since $u_{0}$ is non-negative and the system is mass-conservative, the function $u(t,\cdot)$ is a density for all time $t\geq 0$.
\end{itemize}
\end{prop}

A detailed proof of this result is given in  \ref{sec:proof:prop:characteristics:PPS:linear}.
Here are listed some properties of the solution $u$ of the linear system.

\begin{prop}\label{prop:properties:solution:PPS:linear}
Under the assumptions of Proposition \ref{prop:characteristics:PPS:linear}, assume furthermore that there exists $M>0$ such that for all $s\geq 0,$ $0\leq u^{\rm in}(s)\leq M$.

Then, the solution $u$ of \eqref{eq:edp:PPS:linear} is such that the function $t\mapsto u(t,\cdot)$ belongs to $\mathcal{C}(\mathbb{R}_{+},L^{1}(\mathbb{R}_{+}))$, the function $t\mapsto u(t,0)$ is continuous and
\begin{equation}\label{eq:apriori:bound:u:linear}
0\leq u(t,s)\leq \max(M,|| f ||_{\infty}), \quad \text{for all $t,s\geq 0$.}
\end{equation}
\end{prop}

\begin{proof}
The first continuity property is rather classic thanks to a fixed point argument in the space $\mathcal{C}([0,T],L^{1}(\mathbb{R}_{+}))$ for a good choice of $T>0$  (see \cite[Section 3.3.]{perthame2006transport} for instance).
The second one is given by the second equation of \eqref{eq:edp:PPS:linear} (which is satisfied in a strong sense by the solution given by the characteristics). Indeed,
\begin{eqnarray*}
|u(t+t',0)-u(t,0)| &\leq & \int_{0}^{+\infty} f(t+t',s) |u(t+t' ,s)-u(t,s) | ds \\
& & + \int_{0}^{+\infty} | f(t+t',s) - f(t,s) | u(t,s) ds\\
& \leq & || f ||_{\infty} || u(t+t',\cdot) - u(t,\cdot) ||_{1}\\
&& + \sup_{s\geq 0} | f(t+t',s) - f(t,s) |.
\end{eqnarray*}
Yet, the continuity properties of both functions f and $t\mapsto u(t,\cdot)$ and give that $|u(t+t',0)-u(t,0)|$ goes to $0$ as $t'$ goes to $0$, hence the continuity of $t\mapsto u(t,0)$.

Finally, one can prove that $u$ satisfies \eqref{eq:apriori:bound:u:linear} thanks to the representation given by the characteristics.
On the one hand, the function $u_{0}$ given in Proposition \ref{prop:characteristics:PPS:linear} is non-negative and so is $u$.
On the other hand, it follows from \eqref{eq:characteristics:s>t} that for $s\geq t$, $u(t,s)\leq M$ and it follows from the second equation of \eqref{eq:edp:PPS:linear} that for all $t\geq 0$, $u(t,0)\leq || f ||_{\infty}$ and so \eqref{eq:characteristics:t>s} implies that for $t\geq s$, $u(t,s)\leq || f ||_{\infty}$.
\end{proof}

\subsection{Study of the \eqref{eq:PPS:intro} system}

Here, a global existence result for the nonlinear system \eqref{eq:PPS:intro} is stated under suitable assumptions. Since this result is one of the cornerstone of this work, its proof is given even if its sketch is pretty similar to the proof of \cite[Theorem 5.1]{pakdaman2010dynamics}.

\begin{sloppypar}
In comparison with the uniqueness for the linear system which takes place in $\mathcal{BC}(\mathbb{R}_{+},\mathcal{M}(\mathbb{R}_{+}))$, the uniqueness result stated in this section takes place in $\mathcal{BC}(\mathbb{R}_{+},\mathcal{P}(\mathbb{R}_{+}))$. However, this last result is sufficient for our purpose since it is applied to measures that are probabilities a priori.
\end{sloppypar}

First of all, a technical lemma is needed to fully understand the non linearity involved in the system \eqref{eq:PPS:intro}.

\begin{lem}\label{lem:Xu:well:defined}
Under (\hyperref[ass:Psi:Lipschitz]{$\mathcal{A}^{\Psi}_{\rm Lip}$}) and (\hyperref[ass:Psi:uniformly:bounded]{$\mathcal{A}^{\Psi}_{\infty}$}), assume that $h:\mathbb{R}_{+}\to \mathbb{R}$ is locally integrable and that $f_{0}:\mathbb{R}_{+}\to \mathbb{R}$ is continuous.
Then, for all $u$ in $\mathcal{BC}(\mathbb{R}_{+},\mathcal{P}(\mathbb{R}_{+}))$ there exists a unique function $X_{u}:\mathbb{R}_{+}\mapsto \mathbb{R}$ such that
\begin{equation}\label{eq:fixed:point:Xu}
X_{u}(t)=\int_{z=0}^{t} \int_{s=0}^{+\infty} h(t-z) \Psi\left(s,   X_{u}(z) + f_{0}(z) \right)u\left(z,ds\right)dz.
\end{equation}
Furthermore, the function $X_{u}$ is continuous.
\end{lem}

\begin{proof}
The proof is divided in three steps:

\medskip
-1. Establish a priori estimates on $X_{u}$ to show that it is locally bounded. Indeed, using the the fact that $u$ belongs to $\mathcal{BC}(\mathbb{R}_{+},\mathcal{P}(\mathbb{R}_{+}))$ and the boundedness of $\Psi$, one deduces that
\begin{equation*}
X_{u}(t)\leq ||\Psi||_{\infty} \int_{z=0}^{t} \int_{s=0}^{+\infty} |h(t-z)| u\left(z,ds\right) dz = ||\Psi||_{\infty} \int_{z=0}^{t} |h(t-z)| dz.
\end{equation*}
Hence, the local integrability of $h$ implies the local boundedness of $X_{u}$.

\medskip
-2. Show that $X_{u}$ exists and is unique as a fixed point. For any $T>0$, consider $G_{T}: L^{\infty}([0,T]) \to L^{\infty}([0,T])$ defined, for all $X$ in $L^{\infty}([0,T])$, by
\begin{equation*}
G_{T}(X):=  \left( t\mapsto \int_{z=0}^{t} \int_{s=0}^{+\infty} h(t-z) \Psi\left(s,   X(z) + f_{0}(z) \right)u\left(z,ds\right) dz \right).
\end{equation*}
The Lipschitz continuity of $\Psi$ and the fact that $u(z,\cdot)$ is a probability lead, for any $X_{1},X_{2}$ in $L^{\infty}([0,T])$ and $t$ in $[0,T]$, to
\begin{eqnarray*}
\left| G_{T}(X_{1})(t) -G_{T}(X_{2})(t) \right| &\leq & {\rm Lip}(\Psi) \int_{0}^{t} |h(t-z)| \left| X_{1}(z)-X_{2}(z) \right| dz\\
& \leq & {\rm Lip}(\Psi) || X_{1}-X_{2} ||_{L^{\infty}([0,T])} \int_{0}^{T} |h(z)| dz.
\end{eqnarray*}
Fix $T>0$ such that ${\rm Lip}(\Psi) \int_{0}^{T} h(z) dz \leq 1/2$ so that $G_{T}$ is a contraction and admits a unique fixed point. Iterating this fixed point gives the existence and uniqueness of $X_{u}$ in the space of locally bounded functions. 

For instance, we give the idea for the first iteration. Denoting $W$ the fixed point of $G_{T}$, one can consider $G^{W}_{2T}: L^{\infty}([T,2T]) \to L^{\infty}([T,2T])$ defined, for all $X$ in $L^{\infty}([T,2T])$, by
\begin{equation*}
G^{W}_{2T}(X):=\left(\displaystyle t\mapsto \int_{z=0}^{t} \int_{s=0}^{+\infty} h(t-z) \Psi\left(s,   \tilde{X}(z) + f_{0}(z) \right)u\left(z,ds\right)dz \right),
\end{equation*}
where $\tilde{X}(t)=W(t)$ if $0\leq t <T$, $\tilde{X}(t)=X(t)$ if $T\leq t \leq 2T$,  and $\tilde{X}(t)=0$ otherwise. Applying the same argument as for the fixed point of $G_{T}$ leads to existence and uniqueness of the trace of $X_{u}$ on $[0,2T]$.

\medskip
-3. Finally, let us show that $X_{u}$ is continuous thanks to a generalized Gr\"onwall lemma. Using the Lipschitz continuity and the boundedness of $\Psi$, one deduces from \eqref{eq:fixed:point:Xu} that
\begin{multline*}
\left| X_{u}(t+t') - X_{u}(t) \right|\leq  ||\Psi||_{\infty}  \int_{t}^{t+t'} |h(y)| dy  \\
+  {\rm Lip}(\Psi) \int_{0}^{t} |h(y)| | (X_{u}+f_{0})(t-y) - (X_{u}+f_{0})(t+t'-y)| dy .
\end{multline*}
This means that the function $Y^{(t')}_{u}:= |X_{u}(\cdot+t') - X_{u}(\cdot)|$ satisfies
\begin{equation*}
Y_{u}^{(t')}\leq g^{(t')}(t) + {\rm Lip}(\Psi) \int_{0}^{t} |h(t-z)| Y_{u}^{(t')}(z) dz
\end{equation*}
where $g^{(t')}(t):={\rm Lip}(\Psi) \int_{0}^{t} |h(t-z)| | f_{0}(z+t') - f_{0}(z)| dz+ ||\Psi||_{\infty} \int_{t}^{t+t'} |h(y)| dy $. Applying Lemma \ref{lem:Gronwall:Picard} gives, for any $T>0$, $\sup_{t\in [0,T]} Y_{u}^{(t')} (t)\leq C_{T} \sup_{t\in [0,T]} g^{(t')}(t)$. Yet the continuity (hence uniform continuity on compact time intervals) of $f_{0}$ and the local integrability of $h$ gives that $\sup_{t\in [0,T]} g^{(t')}(t)$ goes to $0$ as $t'$ goes to $0$.
\end{proof}

Now, we have all the ingredients to state the existence/uniqueness result for the \eqref{eq:PPS:intro} system in a measure space of possible solutions.
Notice that the existence/uniqueness result used for the linear system would not directly apply to the non-linear system. In that sense, we extend the result stated in \cite{canizo2013measure}. Let us also mention that our argument is reminiscent of what is called weak-strong uniqueness for measure-valued solutions \cite{brenier2011weak,diperna1987oscillations}: namely prove uniqueness in a large (measure) space and existence in a smaller (smooth) space.

\begin{thm}\label{thm:existence:uniqueness:PPS} 
Under (\hyperref[ass:Psi:Lipschitz]{$\mathcal{A}^{\Psi}_{\rm Lip}$}) and (\hyperref[ass:Psi:uniformly:bounded]{$\mathcal{A}^{\Psi}_{\infty}$}), assume that $h:\mathbb{R}_{+}\to \mathbb{R}$ is locally integrable and that $f_{0}:\mathbb{R}_{+}\to \mathbb{R}$ is continuous. Assume that $u^{\rm in}$ is a non-negative function such that both $\int_{0}^{+\infty} u^{\rm in} (s) ds=1$ and there exists $M>0$ such that for all $s\geq 0,$ $0\leq u^{\rm in}(s)\leq M$. 

Then, there exists a unique solution in the weak sense $u$ such that $t\mapsto u(t,\cdot)$ belongs to $\mathcal{BC}(\mathbb{R}_{+},\mathcal{P}(\mathbb{R}_{+}))$ of the following (PPS) system
\begin{equation}\label{eq:edp:PPS}
\begin{cases} 
\displaystyle \frac{\partial u\left(t,s\right)}{\partial t}+\frac{\partial u\left(t,s\right)}{\partial s} +\Psi\left(s,  X(t)  + f_{0}(t) \right) u\left(t,s\right)=0, \\
\displaystyle u\left(t,0\right)=\int_{s\in \mathbb{R}_{+}} \Psi\left(s,   X(t) + f_{0}(t) \right)u\left(t,s\right)ds ,
\end{cases}
\end{equation}
with initial condition that $u(0,\cdot)=u^{\rm in}$, where for all $t\geq 0$, $X(t)=\int_{0}^{t} h(t-z) u(z,0)dz$.

Moreover, the solution $u$ is such that, for all $t\geq 0$, the measure $u(t,\cdot)$ is a probability and admits a density which is identified to the solution itself. Furthermore, the function $t\mapsto u(t,\cdot)$ belongs to $\mathcal{C}(\mathbb{R}_{+},L^{1}(\mathbb{R}_{+}))$, the function $t\mapsto u(t,0)$ is continuous and
\begin{equation}\label{eq:apriori:bound:u}
0\leq u(t,s)\leq \max(M,|| \Psi ||_{\infty}), \quad \text{for all $t,s\geq 0$.}
\end{equation}
\end{thm}

\begin{rem}
The weak sense means here that for every $\varphi$ in $\mathcal{C}^{\infty}_{c,b}(\mathbb{R}_{+}^{2})$,
\begin{multline}\label{eq:edp:PPS:weak:sense}
\int_{\mathbb{R}_{+}^{2}} \left(\frac{\partial}{\partial t}+\frac{\partial}{\partial s}\right) \varphi\left(t,s\right)u\left( t, ds\right) dt +\int_{\mathbb{R}_{+}} \varphi(0,s) u^{\rm in}(s)ds  \\
+ \int_{\mathbb{R}_{+}^{2}} [\varphi(t,0)-\varphi(t,s) ] \Psi\left(s,   X(t) + f_{0}(t) \right) u(t, ds)dt=0
\end{multline}
where $X$ is the continuous function given by Lemma \ref{lem:Xu:well:defined} and satisfying
\begin{equation}\label{eq:fixed:point:X}
X(t)=\int_{z=0}^{t} \int_{s=0}^{+\infty} h(t-z) \Psi\left(s,   X(z) + f_{0}(z) \right)u\left(z,ds\right)dz.
\end{equation}
\end{rem}

As for the linear case, the system is mass-conservative (e.g. take a sequence of functions converging to $t\mapsto \mathds{1}_{[0,T]}(t)$ as test functions in the weak equation \eqref{eq:edp:PPS:weak:sense}).

\begin{proof}
The proof is divided in two steps. First, we apply the results of Section \ref{sec:linear:system} to a linearised version of the non-linear system \eqref{eq:edp:PPS} and then we find the auxiliary function $X$ corresponding to the solution $u$ as a fixed point in a space of continuous functions in order to deal with the non linearity of the system \eqref{eq:edp:PPS}.

\medskip
-1. The linearised version of the system takes the form
\begin{equation}\label{eq:edp:PPS:frozen}
\begin{cases} 
\displaystyle \frac{\partial u\left(t,s\right)}{\partial t}+\frac{\partial u\left(t,s\right)}{\partial s} +\Psi\left(s,  Y(t)  + f_{0}(t) \right) u\left(t,s\right)=0, \\
\displaystyle u\left(t,0\right)=\int_{s\in \mathbb{R}_{+}} \Psi\left(s,   Y(t) + f_{0}(t) \right)u\left(t,s\right)ds ,
\end{cases}
\end{equation}
for a fixed continuous function $Y$. Note that the function $f:(t,s)\mapsto \Psi\left(s,   Y(t) + f_{0}(t) \right)$ is bounded and continuous (uniformly in $s$)  with respect to $t$. So the assumptions of Propositions \ref{prop:existence:uniqueness:PPS:linear}, \ref{prop:characteristics:PPS:linear} and \ref{prop:properties:solution:PPS:linear} are satisfied. In particular, for any continuous function $Y$, there exists a unique solution $u_{Y}$ (with initial condition $u^{\rm in}$) in $\mathcal{BC}(\mathbb{R}_{+},\mathcal{P}(\mathbb{R}_{+}))\subset \mathcal{BC}(\mathbb{R}_{+},\mathcal{P}(\mathbb{M}_{+}))$ of the system \eqref{eq:edp:PPS:frozen} which furthermore satisfies the properties listed in Proposition \ref{prop:properties:solution:PPS:linear}.

\medskip
-2. Let us notice that for all $T>0$, if $Y$ belongs to the Banach space $(\mathcal{C}([0,T]), ||.||_{\infty,[0,T]})$ then $t\mapsto \int_{0}^{t} h(t-z) u_{Y}(z,0)dz$ belongs to $\mathcal{C}([0,T])$ too. Indeed, recall that $t\mapsto u_{Y}(t,0)$ is continuous thanks to Proposition \ref{prop:properties:solution:PPS:linear} and that $\int_{0}^{t} h(t-z) u_{Y}(z,0)dz=\int_{0}^{t} h(z) u_{Y}(t-z,0)dz$.
Hence one can consider the map
\begin{equation}\label{eq:def:FT}
\begin{array}{cccc}
F_{T}:\ & \mathcal{C}([0,T]) &\longrightarrow& \mathcal{C}([0,T])\\
 & Y &\longmapsto &\left(\displaystyle t\mapsto \int_{0}^{t} h(t-z) u_{Y}(z,0)dz \right),
\end{array}
\end{equation}
and show that it admits a fixed point for a good choice of $T$.
Computations given in  \ref{sec:proof:thm:existence:uniqueness:PPS} provide the following statement
\begin{equation}\label{eq:F:contraction}
\exists\, T>0, \forall\, Y_{1},Y_{2}\in L^{\infty}([0,T]), 
\quad || F_{T}(Y_{1}) - F_{T}(Y_{2}) ||_{L^{\infty}([0,T])} \leq \frac{1}{2} || Y_{1}-Y_{2} ||_{L^{\infty}([0,T])},
\end{equation}
where $T$ depends neither on the initial condition nor on $f_{0}$. 

Until the end of the proof, let fix such a $T$. Then, there exists a unique $W$ in $\mathcal{C}([0,T])$ such that $F_{T}(W)=W$. In particular, $u_{W}$ is a solution of \eqref{eq:edp:PPS} on $[0,T]$ so we have existence on $[0,T]$.

For the uniqueness, let us consider the trace (on $[0,T]$) of a solution $u\in \mathcal{BC}(\mathbb{R}_{+},\mathcal{P}(\mathbb{R}_{+}))$ of \eqref{eq:edp:PPS}. Then, the auxiliary function $X_{u}$ associated with $u$ defined in Lemma \ref{lem:Xu:well:defined} is continuous. Since $u$ is a solution of \eqref{eq:edp:PPS}, the trace of $X_{u}$ on $[0,T]$ is a fixed point of $F_{T}$ and so $X_{u}=W$ and $u=u_{W}$. This gives the uniqueness of the solution of \eqref{eq:edp:PPS} on $[0,T]$ in $\mathcal{BC}(\mathbb{R}_{+},\mathcal{P}(\mathbb{R}_{+}))$.

Taking $u_{W}(T,\cdot)$ instead of $u^{\rm in}$ as initial condition, the function $t\mapsto f_{0}(t+T) + \int_{0}^{T} h(t-z)u_{W}(z,0) dz$ instead of $f_{0}$ and applying the same kind of fixed point argument gives the trace of the solution on $[T,2T]$. Iterating this fixed point argument, one deduces that there exists a unique solution $u$ of \eqref{eq:edp:PPS} on $\mathbb{R}_{+}$ (remind that $T$ depends neither on the initial condition nor on $f_{0}$). In particular, the iteration is possible since the boundedness of the initial condition is carried on by the equation (see Equation \eqref{eq:apriori:bound:u:linear}).

The regularity and boundedness of the solution $u$, i.e. the continuity properties and Equation \eqref{eq:apriori:bound:u} listed at the end of the statement, come from the regularity and boundedness of the solutions $u_{Y}$ since $u$ is one of the $u_{Y}$'s.

\end{proof}

\subsection{Limit process}

The limit equation \eqref{eq:limit:equation:counting:process} describes an age dependent point process interacting with its own mean intensity in an Hawkes manner. More precisely, a solution $(\overline{N}_{t})_{t\geq 0}$ of \eqref{eq:limit:equation:counting:process}, if it exists, admits an intensity $\overline{\lambda}_{t}$ which depends on the time $t$ and the age $\overline{S}_{t-}$ in a McKean-Vlasov manner in the sense that it satisfies the following implicit equation $\overline{\lambda}_{t}=\Psi(\overline{S}_{t-},  \int_{0}^{t'-} h(t'-z) \mathbb{E} \left[\overline{\lambda}_{z} \right] dz + f_{0}(t') )$. Equation \eqref{eq:limit:equation:counting:process} is in particular a non trivial fixed point problem of the McKean-Vlasov type.
Notice that the dependence of $\overline{N}$ with respect to the past $\overline{N}_{-}$ is reduced to the age at time $0$, i.e. $\overline{S}_{0-}=-T_{0}$.
Throughout this article, a solution of the limit equation is called a \emph{point process of the McKean-Vlasov type whose intensity depends on time and on the age}.

The fixed point problem of the limit equation \eqref{eq:limit:equation:counting:process} is proved to be well-posed in two cases as given in the next two statements. In either case, the idea of the proof is first to compute the mean intensity (denoted by $\overline{\lambda}(t)$) of a possible solution of \eqref{eq:limit:equation:counting:process}. The next two propositions state the same result under different sets of assumptions and can be summarized as follows:
\begin{itemize}
\item in the first case, the intensity is bounded, i.e. $\Psi$ satisfies (\hyperref[ass:Psi:uniformly:bounded]{$\mathcal{A}^{\Psi}_{\infty}$}), and the mean intensity is given by the system \eqref{eq:edp:PPS}. More precisely, $\overline{\lambda}(t)=u(t,0)$  with $u$ given by Theorem \ref{thm:existence:uniqueness:PPS}.
\item in the second case, the intensity does not depend on the age process, i.e. $\Psi$ satisfies (\hyperref[ass:Psi:=Psi0]{$\mathcal{A}_{\Psi=\Psi_{0}}$}), and the mean intensity is given by a generalization of \cite[Lemma 24]{delattre2016}.
\end{itemize}

\begin{prop}\label{prop:limit:equation:well:posed:bounded:intensity}
Under (\hyperref[ass:initial:condition:density]{$\mathcal{A}^{\zeta_{N_{-}}}_{u^{\rm in}}$}), (\hyperref[ass:Psi:Lipschitz]{$\mathcal{A}^{\Psi}_{\rm Lip}$}) and (\hyperref[ass:Psi:uniformly:bounded]{$\mathcal{A}^{\Psi}_{\infty}$}), assume that $h:\mathbb{R}_{+}\to \mathbb{R}$ is locally integrable and that $f_{0}:\mathbb{R}_{+}\to \mathbb{R}$ is continuous. Denote by $u$ the unique solution of \eqref{eq:edp:PPS} with initial condition $u^{\rm in}$ as given by Theorem \ref{thm:existence:uniqueness:PPS} and let, for all $t\geq 0$, $\overline{\lambda}(t):=u(t,0)$. Then, the following statements hold
\begin{enumerate}[(i)]
\item if $(\overline{N}_t)_{t\geq 0}$ is a solution of \eqref{eq:limit:equation:counting:process} then $\mathbb{E} \left[\overline{N}_{+}(dt) \right]=\overline{\lambda}(t)dt$,
\item there exists a unique (once $\Pi$ and $\overline{N}_{-}$ are fixed) solution $(\overline{N}_t)_{t\geq 0}$ of the following system
\begin{equation}\label{eq:limit:equation:counting:process:rewritten:PPS}
\left\{
\begin{aligned}
&\overline{N}_t= \int_0^t \int_0^\infty \mathds{1}_{\Big\{ x\leq \Psi\left( \overline{S}_{t'-},  \int_{0}^{t'} h(t'-z) \overline{\lambda}(z)dz + f_{0}(t') \right)  \Big\}}  \, \Pi(dt',dx), \\
&\mathbb{E}\left[ \overline{N}_{t} \right] = \int_{0}^{t} \overline{\lambda}(t') dt',
\end{aligned}
\right.
\end{equation}
where $(\overline{S}_{t-})_{t\geq 0}$ is the predictable age process associated with $\overline{N}=\overline{N}_{-}\cup \overline{N}_{+}$ where $\overline{N}_{-}$ is a point process distributed according to $\zeta_{N_{-}}$ and $\overline{N}_{+}$ is the point process associated with the counting process $(\overline{N}_t)_{ t \geq 0}$
\end{enumerate}
In particular, $\overline{\lambda}$ is a continuous function satisfying 
\begin{equation}\label{eq:mckean:vlasov:intensity}
\overline{\lambda}(t)=\mathbb{E}\left[ \Psi\left( \overline{S}_{t-},  \int_{0}^{t} h(t-z) \overline{\lambda}(z)dz + f_{0}(t) \right)  \right]
\end{equation}
and the solution of \eqref{eq:limit:equation:counting:process:rewritten:PPS} is the unique (once $\Pi$ and $\overline{N}_{-}$ are fixed) solution of \eqref{eq:limit:equation:counting:process}.
\end{prop}

\begin{proof}
-$(i)$ Suppose that $(\overline{N}_t)_{t\geq 0}$ is a solution of \eqref{eq:limit:equation:counting:process}. The thinning procedure implies that $(\overline{N}_t)_{t\geq 0}$ admits an intensity which only depends on the time $t$ and the age $\overline{S}_{t-}$. This allows us to denote by $f$ the bivariate function such that the intensity of $\overline{N}$ at time $t$ is given by $\overline{\lambda}_{t}=f(t,\overline{S}_{t-})$. It satisfies for all $t,s\geq 0$,
\begin{equation}\label{eq:intensity:limit:equation:time:age:expectation}
f(t,s) = \Psi\left( s,  \int_{0}^{t-} h(t-z) \mathbb{E}\left[ f(z,\overline{S}_{z-}) \right]dz + f_{0}(t) \right).
\end{equation}
In particular, the intensity is bounded since $\Psi$ is bounded. So, if we denote $w_{t}=w(t,\cdot)$ the distribution of the age $\overline{S}_{t-}$, Lemma \ref{lem:continuity:age:process} gives that $w$ belongs to $\mathcal{BC}(\mathbb{R}_{+},\mathcal{P}(\mathbb{R}_{+}))$ and \cite[Section 4.1]{chevallier2015microscopic} implies that $w$ satisfies the system
\begin{equation*}
\begin{cases}
\displaystyle \frac{\partial w\left(t,s\right)}{\partial t}+
\frac{\partial w\left(t,s\right)}{\partial s}+
f(t,s) w\left(t,s\right)=0,\\
\displaystyle w\left(t,0\right)=
\int_{0}^{+\infty} f(t,s) w\left(t,s\right)ds.
\end{cases}
\end{equation*}
Yet, by definition of $w$, $\mathbb{E}\left[ f(z,\overline{S}_{z-}) \right]=\int_{s=0}^{+\infty} f(z,s) w(z,ds)$, so \eqref{eq:intensity:limit:equation:time:age:expectation} rewrites as
\begin{equation}\label{eq:intensity:limit:equation:time:age:w}
f(t,s) = \Psi\left( s,  X(t)  + f_{0}(t) \right),
\end{equation}
where $X$ satisfies
\begin{equation*}
X(t)= \int_{0}^{t} h(t-z) \int_{0}^{+\infty} \Psi(s,X(z)+f_{0}(z))w(z,ds) dz.
\end{equation*}
Hence, $w$ is a solution in $\mathcal{BC}(\mathbb{R}_{+},\mathcal{P}(\mathbb{R}_{+}))$ of the system \eqref{eq:edp:PPS}. Yet, the solution of \eqref{eq:edp:PPS} is unique (Theorem \ref{thm:existence:uniqueness:PPS}) so we have $w=u$ (defined in Proposition \ref{prop:limit:equation:well:posed:bounded:intensity}) and in particular $\mathbb{E} \left[\overline{N}_{+}(dt) \right]=\mathbb{E}\left[ f(t,\overline{S}_{t-}) \right]dt=u(t,0)dt=\overline{\lambda}(t)dt$.

\medskip
-$(ii)$ The first equation of \eqref{eq:limit:equation:counting:process:rewritten:PPS} is a standard thinning equation with $\overline{\lambda}$ given by the first step so its solution $(\overline{N}_{t})_{t\geq 0}$ is a measurable function of $\Pi$ and $\overline{N}_{-}$ hence it is unique (once $\Pi$ and $\overline{N}_{-}$ are fixed).

To conclude  this step, it suffices to check that $(\overline{N}_{t})_{t\geq 0}$ satisfies the second equation of \eqref{eq:limit:equation:counting:process:rewritten:PPS}. Identifying $\overline{\lambda}(t)$ with $\mathbb{E}\left[ f(t,\overline{S}_{t-}) \right]$ in \eqref{eq:intensity:limit:equation:time:age:expectation}, the intensity of $\overline{N}$ is given by 
\begin{equation}\label{eq:intensity:limit:equation:time:age:lambda:barre}
f(t,\overline{S}_{t-}) = \Psi\left( \overline{S}_{t-},  \int_{0}^{t-} h(t-z) \overline{\lambda}(z)dz + f_{0}(t) \right)
\end{equation}
which is bounded.
Hence, \cite[Section 4.1]{chevallier2015microscopic} implies that the distribution of the age $\overline{S}_{t-}$ denoted by $v(t,\cdot)$ is the unique solution of
\begin{equation}\label{eq:edp:PPS:pre}
\begin{cases}
\displaystyle \frac{\partial v\left(t,s\right)}{\partial t}+
\frac{\partial v\left(t,s\right)}{\partial s}+
\Psi\left(s,  \int_{0}^{t} h(t-z) \overline{\lambda}(z) dz  + f_{0}(t) \right) v\left(t,s\right)=0,\\
\displaystyle v\left(t,0\right)=
\int_{0}^{\infty} \Psi\left(s,  \int_{0}^{t} h(t-z) \overline{\lambda}(z) dz  + f_{0}(t) \right) v\left(t,s\right)ds.
\end{cases}
\end{equation}
Since $\overline{\lambda}(t)=u(t,0)$ and $u$ is a solution of \eqref{eq:edp:PPS}, it is clear that $u$ satisfies this system, so $u(t,\cdot)$ is the density of $\overline{S}_{t-}$.
Finally, using Fubini's Theorem we have
\begin{eqnarray*}
\mathbb{E}\left[ \overline{N}_{t} \right] &=& \int_0^t \mathbb{E}\left[ \Psi\left( \overline{S}_{t'-},  \int_{0}^{t'-} h(t'-z) \overline{\lambda}(z) dz + f_{0}(t') \right)  \right] dt'\\
 & = & \int_0^t \int_{0}^{\infty}  \Psi\left( s,  \int_{0}^{t'} h(t'-z) \overline{\lambda}(z) dz + f_{0}(t') \right) u(t',s) ds dt'\\
 & = & \int_{0}^{t'} \overline{\lambda}(t') dt',
\end{eqnarray*}
since $u$ satisfies the second equation of \eqref{eq:edp:PPS:pre}.\\

Finally, the three remaining points are rather simple. Firstly, the continuity of $\overline{\lambda}$ comes from Theorem \ref{thm:existence:uniqueness:PPS}.
Secondly, using \eqref{eq:intensity:limit:equation:time:age:lambda:barre} and $(i)$ one has
$$\overline{\lambda}(t)=\mathbb{E}\left[ f(t,\overline{S}_{t-}) \right]=\mathbb{E}\left[ \Psi\left( \overline{S}_{t-},  \int_{0}^{t-} h(t-z) \overline{\lambda}(z)dz + f_{0}(t) \right)  \right].$$ 
Lastly, the solution of \eqref{eq:limit:equation:counting:process:rewritten:PPS} is clearly a solution of \eqref{eq:limit:equation:counting:process} and $(i)$ tells that a solution of \eqref{eq:limit:equation:counting:process} is necessarily a solution of \eqref{eq:limit:equation:counting:process:rewritten:PPS} which gives uniqueness.
\end{proof}

\begin{prop}\label{prop:limit:equation:well:posed:delta=0}
Under (\hyperref[ass:Psi:Lipschitz]{$\mathcal{A}^{\Psi}_{\rm Lip}$}) and (\hyperref[ass:Psi:=Psi0]{$\mathcal{A}_{\Psi=\Psi_{0}}$}) assume that $h:\mathbb{R}_{+}\to \mathbb{R}$ is locally integrable and that $f_{0}:\mathbb{R}_{+}\to \mathbb{R}$ is continuous. 

Then, there exists a unique function $\overline{\lambda}$ (which is furthermore continuous on $\mathbb{R}_{+}$) depending only on $\Psi_{0}$, $h$ and $f_{0}$ such that the following statements hold
\begin{enumerate}[(i)]
\item if $(\overline{N}_t)_{t\geq 0}$ is a solution of \eqref{eq:limit:equation:counting:process} then $\mathbb{E} \left[\overline{N}_{+}(dt) \right]=\overline{\lambda}(t)dt$,
\item there exists a unique (once $\Pi$ is fixed) solution $(\overline{N}_t)_{t\geq 0}$ to the following system
\begin{equation}\label{eq:limit:equation:counting:process:rewritten:delta=0}
\left\{
\begin{aligned}
&\overline{N}_t= \int_0^t \int_0^\infty \mathds{1}_{\Big\{ x\leq \Psi_{0}\left( \int_{0}^{t'} h(t'-z) \overline{\lambda}(z)dz + f_{0}(t') \right)  \Big\}}  \, \Pi(dt',dx), \\
&\mathbb{E}\left[ \overline{N}_{t} \right] = \int_{0}^{t} \overline{\lambda}(t') dt'.
\end{aligned}
\right.
\end{equation}
\end{enumerate}
In particular, $\overline{\lambda}(t)=\Psi_{0}\left(  \int_{0}^{t-} h(t-z) \overline{\lambda}(z)dz + f_{0}(t) \right) $ and the solution of \eqref{eq:limit:equation:counting:process:rewritten:delta=0} is the unique (once $\Pi$ is fixed) solution of \eqref{eq:limit:equation:counting:process}.
\end{prop}
The proof follows \cite[Theorem 8-{\it (i)}]{delattre2016} and is given in  \ref{sec:proof:prop:limit:equation} for sake of exhaustiveness.

\subsection{Link via the age process}

The link between the limit equation \eqref{eq:limit:equation:counting:process} and system \eqref{eq:edp:PPS} is even deeper than what is stated in Proposition \ref{prop:limit:equation:well:posed:bounded:intensity}. Indeed, the distribution of the age process (either the predictable or the standard one) associated with a solution of the limit equation is a solution of \eqref{eq:edp:PPS} as described in the next statement.

\begin{prop}\label{prop:age:process:limit:equation:PPS}
Under the assumptions of Proposition \ref{prop:limit:equation:well:posed:bounded:intensity}, the unique solution $u$ to the system \eqref{eq:edp:PPS} with initial condition that  $u(0,\cdot)=u^{\rm in}$ is such that $u(t,\cdot)$ is the density of the age $\overline{S}_{t-}$ (or $\overline{S}_{t}$ since they are equal a.s.) associated with the solution of the limit equation \eqref{eq:limit:equation:counting:process} given in Proposition \ref{prop:limit:equation:well:posed:bounded:intensity}. 
\end{prop}
This result is in fact given in the proof of Proposition \ref{prop:limit:equation:well:posed:bounded:intensity} below Equation \eqref{eq:edp:PPS:pre}.

\section{Mean-field dynamics}
\label{sec:mean:field}
The convergence to the limit dynamics is proved by using a path-wise coupling (like in \cite{delattre2016,Meleard_96}) between the processes given by Representation \ref{def:Hawkes:Thinning} on the one hand and by the limit equation on the other hand. Then, this coupling is studied in two different cases: when the intensity is bounded, i.e. $\Psi$ satisfies (\hyperref[ass:Psi:uniformly:bounded]{$\mathcal{A}^{\Psi}_{\infty}$}), or when the intensity does not depend on the age process, i.e. $\Psi$ satisfies (\hyperref[ass:Psi:=Psi0]{$\mathcal{A}_{\Psi=\Psi_{0}}$}).
 
The precise statement regarding the convergence of the $n$-particle system towards point processes of the McKean-Vlasov type whose intensity depends on time and on the age is given in Corollary \ref{cor:mean:field:approx:counting:and:age:process}.
 
\subsection{Coupling}
Once the limit equation is well-posed, following the ideas of Sznitman in \cite{Sznitman_91}, it is easy to construct a suitable coupling between ADRHPs and i.i.d. solutions of the limit equation \eqref{eq:limit:equation:counting:process}. More precisely, consider
$\bullet$ a sequence $(N_{-}^{i})_{i\geq 1}$ of i.i.d. point processes distributed according to $\zeta_{N_{-}}$;

$\bullet$ an infinite matrix $(H_{ij})_{i,j\geq 1}$ (independent of $(N_{-}^{i})_{i\geq 1}$) with entries distributed according to $\mu_{H}$ such that 
\begin{equation}\label{eq:indep:Hij:infinite}
{\small
\begin{cases}
\text{for any fixed $i\geq 1$, the variables $H_{i1},\dots ,H_{in},\dots$ are independent,}\\
\text{the sequences $H_{i1},\dots ,H_{in},\dots$ are exchangeable (with respect to $i$),}\\
\text{the matrix $(H_{ij})_{i,j\geq 1}$ is independent from $(N_{-}^{i})_{i=1,..,n}$;}
\end{cases}
}
\end{equation}

$\bullet$ an infinite matrix $(F_{ij})_{i,j\geq 1}$ with entries distributed according to $\nu_{F}$ such that 
\begin{equation}\label{eq:indep:Fij:infinite}
{\small
\begin{cases}
\text{for any fixed $i\geq 1$, the variables $F_{i1},\dots ,F_{in},\dots$ are independent,}\\
\text{the sequences $F_{i1},\dots ,F_{in},\dots$ are exchangeable (with respect to $i$),}
\end{cases}
}
\end{equation}

$\bullet$ a sequence $(\Pi^i(dt',dx))_{i\geq 1}$ of i.i.d. $\mathbb{F}$-Poisson measures with intensity $1$ on $\mathbb{R}_{+}^{2}$.
\medskip

\noindent Notice that \eqref{eq:indep:Hij:infinite} (resp. \eqref{eq:indep:Fij:infinite}) is the equivalent of \eqref{eq:indep:Hij} (resp. \eqref{eq:indep:Fij}) for infinite matrices.
Under (\hyperref[ass:mu:H:locally:square:integrable]{$\mathcal{A}^{\mu_{H}}_{\infty,2}$}), (\hyperref[ass:nu:F:variance]{$\mathcal{A}^{\nu_{F}}_{2}$}) and (\hyperref[ass:Psi:Lipschitz]{$\mathcal{A}^{\Psi}_{\rm Lip}$}), the assumptions of Proposition \ref{prop:Thinning:well:posed} are satisfied. Furthermore, if we assume either:
\begin{itemize}
\nsp\nsp
\item $\mathcal{H}_{1}$: ``(\hyperref[ass:Psi:uniformly:bounded]{$\mathcal{A}^{\Psi}_{\infty}$}) and (\hyperref[ass:initial:condition:density]{$\mathcal{A}^{\zeta_{N_{-}}}_{u^{\rm in}}$}) are satisfied", \phantomsection \label{ass:H1}
\item $\mathcal{H}_{2}$: ``(\hyperref[ass:Psi:=Psi0]{$\mathcal{A}_{\Psi=\Psi_{0}}$}) is satisfied", \phantomsection  \label{ass:H2}
\end{itemize}
then $m_{\mu_{H}}$ and $m_{\nu_{F}}$ (defined in (\hyperref[ass:mu:H:infty]{$\mathcal{A}^{\mu_{H}}_{\infty}$}) and (\hyperref[ass:nu:F:locally:bounded]{$\mathcal{A}^{\nu_{F}}_{1}$})) satisfy the assumptions of either Proposition \ref{prop:limit:equation:well:posed:bounded:intensity} (under \hyperref[ass:H1]{$\mathcal{H}_{1}$}) or \ref{prop:limit:equation:well:posed:delta=0} (under \hyperref[ass:H2]{$\mathcal{H}_{2}$}) so one can build simultaneously:
\smallskip

 - a sequence\footnote{The sequence is indexed by $n\geq 1$.} $(N^{n,i})_{i=1,\dots ,n}$ of ADRHPs with parameters $(n,\mu_{H},\nu_{F},\Psi,\zeta_{N_{-}})$ according to Representation~\ref{def:Hawkes:Thinning} that is
\begin{equation}\label{eq:definition:coupling:ADRHP}
N^{n,i}_t= \int_0^t \int_0^\infty \mathds{1}_{  \left\{x \leq \Psi \left(S^{n,i}_{t'-},\frac{1}{n}  \sum_{j=1}^n \left[\int_{0}^{t'-} H_{ij}(t'-z)N^{n,j}_{+}(dz) + F_{ij}(t') \right] \right) \right\}} \Pi^i(dt',dx)
\end{equation}
with predictable age processes $(S^{n,i}_{t-})_{t\geq 0}^{i=1,\dots,n}$ and past given by $N^{i}_{-}$,

 - and a sequence $(\overline{N}^i_t)^{i\geq 1}_{t\geq 0}$ of i.i.d. solutions of the limit equation with parameters $(m_{\mu_{H}},m_{\nu_{F}},\Psi,\zeta_{N_{-}})$ that is
\begin{equation}\label{eq:definition:coupling:limit:equation}
\overline{N}^{i}_t= \int_0^t \int_0^\infty \mathds{1}_{ \left\{x \leq \Psi \left( \overline{S}^{i}_{t'-},\int_0^{t'} m_{\mu_{H}}(t'-z)\overline{\lambda}(z) dz +m_{\nu_{F}}(t') \right)\right\}} \Pi^i(dt',dx),
\end{equation}
where $\overline{\lambda}$ is defined either in Proposition \ref{prop:limit:equation:well:posed:bounded:intensity} (under \hyperref[ass:H1]{$\mathcal{H}_{1}$}) or \ref{prop:limit:equation:well:posed:delta=0} (under \hyperref[ass:H2]{$\mathcal{H}_{2}$}) and $(\overline{S}^{i}_{t-})_{t\geq 0}$ is the predictable age process associated with $\overline{N}^{i}=N^{i}_{-}\cup \overline{N}^{i}_{+}$.
\smallskip

Notice that this coupling is based on the sharing of a common past $(N_{-}^{i})_{i\geq 1}$ and a common underlying randomness, that are the $\mathbb{F}$-Poisson measures $(\Pi^i(dt',dx))_{i\geq 1}$.
Note that the sequence of ADRHPs is indexed by the size of the network $n$ whereas the solutions of the limit equation that represent the behaviour under the mean-field approximation are not.

The following result states the control of the mean-field approximation.

\begin{thm}\label{thm:counting:process:convergence}
Under (\hyperref[ass:mu:H:locally:square:integrable]{$\mathcal{A}^{\mu_{H}}_{\infty,2}$}), (\hyperref[ass:nu:F:variance]{$\mathcal{A}^{\nu_{F}}_{2}$}) and (\hyperref[ass:Psi:Lipschitz]{$\mathcal{A}^{\Psi}_{\rm Lip}$}), assume either \hyperref[ass:H1]{$\mathcal{H}_{1}$} or  \hyperref[ass:H2]{$\mathcal{H}_{2}$}.

Then, the sequence of ADRHP $(N^{n,i})_{i=1,..,n}$ (with $\mathbb{F}$-intensities on $\mathbb{R}_{+}$ denoted by $(\lambda^{n,i})_{i=1,..,n}$) defined by \eqref{eq:definition:coupling:ADRHP} and the i.i.d. copies $(\overline{N}^i_t)^{i\geq 1}_{t\geq 0}$ of the solution of the limit equation (with $\mathbb{F}$-intensities on $\mathbb{R}_{+}$ denoted by $(\overline{\lambda}^{i})_{i=1,..,n}$) defined by \eqref{eq:definition:coupling:limit:equation} are such that for all $i=1,\dots ,n$ and $\theta>0$,
\begin{equation}\label{eq:bound:full:expect:counting:process}
\mathbb{E}\left[ \sup_{t\in [0,\theta]} | N^{n,i}_t -\overline{N}^i_t | \right] \leq  \int_{0}^{\theta} \mathbb{E}\left[| \lambda^{n,i}_t - \overline{\lambda}^{i}_{t} | \right]dt  \leq  
C(\theta,\Psi,\mu_{H},\nu_{F}) \, n^{-1/2},
\end{equation}
where the constant $C(\theta,\Psi,\mu_{H},\nu_{F})$ does not depend on $n$.
\end{thm}

\begin{proof}
First, let us note that in each cases, the coupling is well-defined thanks to either Proposition \ref{prop:limit:equation:well:posed:bounded:intensity} or \ref{prop:limit:equation:well:posed:delta=0}.
Let us fix some $n\geq 1$. Let us denote $\Delta^{i}_{n}(t)= \int_{0}^{t} |N^{n,i}_{+}(dt')-\overline{N}^i_{+}(dt')|$ and $\delta^{i}_n(t)=\mathbb{E}[\Delta^{i}_{n}(t)]$ its expectation. Denoting $A \triangle B$ the symmetric difference of the sets $A$ and $B$ and $\card(A)$ the cardinal of the set $A$ we have
\begin{equation}\label{eq:def:Delta:i}
\Delta^{i}_{n}(t)= \int_{0}^{t} |N^{n,i}_{+}(dt')-\overline{N}_{+}^i(dt')|=\card \left( ( N^{n,i}\triangle\, \overline{N}^{i} )  \cap [0,t]\right) ,
\end{equation}
that is the number of points that are not common to $N^{n,i}$ and $\overline{N}^i$ between 0 and $t$.
Then, it is clear that, for all $i=1,\dots ,n$ and $\theta>0$, $\sup_{t\in [0,\theta]} |N^{i}_t - \overline{N}^i_t| \leq \Delta^{i}_n(\theta)$  and so
\begin{equation}\label{eq:bound:L1:coupling}
\mathbb{E}\left[ \sup_{t\in [0,\theta]} |N^{n,i}_t - \overline{N}^i_t| \right] \leq \mathbb{E}\Big[ \Delta_{n}^{i}(\theta) \Big] = \delta^{i}_n(\theta).
\end{equation}
On the one hand, the $\overline{N}^{i}$'s are i.i.d. hence exchangeable. On the other hand, thanks to the form of the intensity and the assumptions on the matrices $(H_{ij})_{i,j\geq 1}$ and $(F_{ij})_{i,j\geq 1}$ - \eqref{eq:indep:Hij:infinite} and \eqref{eq:indep:Fij:infinite} - the family $(N^{n,i})_{i=1,\dots,n}$ is exchangeable too. 
Hence $\delta^{i}_n$ does not depend on $i$ and is simply denoted by $\delta_{n}$ in the sequel. Let us focus on the case $i=1$.
First, let us remind that $\overline{\lambda}^{1}_{t}$ is the intensity of $\overline{N}^{1}$, so
$$
\Delta^{1}_{n}(\theta)= \int_0^\theta \int_0^\infty \Big| \mathds{1}_{\big\{x \leq \lambda^{n,1}_{t} \big\}} - \mathds{1}_{\big\{x \leq \overline{\lambda}^{1}_{t} \big\}} \Big| \, \Pi^1(dt,dx).
$$
Taking expectation we find
\begin{eqnarray}
\delta_n(\theta)&=& \mathbb{E}\left[ \int_0^\theta \mathbb{E}\left[ \int_0^\infty \Big| \mathds{1}_{\big\{x \leq \lambda^{n,1}_{t} \big\}} - \mathds{1}_{\big\{x \leq \overline{\lambda}^{1}_{t} \big\}} \Big| \, \Pi^1(dt,dx) \middle|\, \mathcal{F}_{t-} \right] \right] \nonumber \\
& = & \mathbb{E}\left[ \int_0^\theta  \int_0^\infty \Big| \mathds{1}_{\big\{x \leq \lambda^{n,1}_{t} \big\}} - \mathds{1}_{\big\{x \leq \overline{\lambda}^{1}_{t} \big\}} \Big| dx dt  \right] \nonumber\\
&=& \mathbb{E}\left[  \int_{0}^{\theta} | \lambda^{n,1}_t - \overline{\lambda}^{1}_{t} | dt \right] = \int_{0}^{\theta} \mathbb{E}\left[\Big| \lambda^{n,1}_{t} - \overline{\lambda}^{1}_{t} \Big|  \right] dt,\label{eq:delta:n:intensity}
\end{eqnarray}
where the last equality comes from Fubini's Theorem. It remains to show that the rate of convergence is $n^{-1/2}$.

Computations given in Section \ref{sec:proof:thm:counting:process:convergence} show that in each cases there exists some locally bounded function $g$ depending on $\Psi$, $\mu_{H}$ and $\nu_{F}$ such that $\delta_{n}$ satisfies,
\begin{equation}\label{eq:bound:delta:both:cases}
\left\{
\begin{aligned}
& \delta_{n}(\theta)\leq n^{-1/2} g(\theta) +  \int_{0}^{\theta} \left[ || \Psi ||_{\infty} + {\rm Lip}(\Psi) M_{\mu_{H}}(\theta-z)\right] \delta_n(z) dz && \text{(under \hyperref[ass:H1]{$\mathcal{H}_{1}$})}, \\
& \delta_{n}(\theta)\leq n^{-1/2} g(\theta)  +  \int_{0}^{\theta} {\rm Lip}(\Psi) M_{\mu_{H}}(\theta-z) \delta_n(z) dz && \text{(under \hyperref[ass:H2]{$\mathcal{H}_{2}$})}.
\end{aligned}
\right.
\end{equation}
\begin{sloppypar}
Remark that the only dependence with respect to $n$ lies in $\delta_{n}$. 
Since $g$ is locally bounded and $M_{\mu_{H}}$ is locally integrable, using Lemma \ref{lem:Gronwall:Picard}-{\it (i)}, we end up with $\delta_{n}(\theta)\leq C(\theta,\Psi,\mu_{H},\nu_{F}) \, n^{-1/2}$ where $C$ does not depend on $n$.
\end{sloppypar}
\end{proof}

There are mainly two reasons for the dichotomy of the assumptions \hyperref[ass:H1]{$\mathcal{H}_{1}$} and \hyperref[ass:H2]{$\mathcal{H}_{2}$}. Firstly, up to our knowledge, existence/uniqueness results on the macroscopic system \eqref{eq:edp:PPS} are only valid if the function $\Psi$ is bounded. Secondly, as it appears in \eqref{eq:bound:delta:both:cases}, when the intensity of the $n$-particle system depends on the age, the control of $\delta_{n}(\theta)$ involves the $L^{\infty}$ norm of the function $\Psi$. This boundedness condition is used in order to control the coupling as soon as the ages of the $n$-particle system on the one hand and the i.i.d. copies of the solution of the limit equation on the other hand are different.
Notice that even under \hyperref[ass:H2]{$\mathcal{H}_{2}$}, this coupling result extends \cite[Theorem 8-{\it (ii)}]{delattre2016} since there are two novelties in the present article: random interaction functions $H_{ij}$ as well as dependences with respect to the past $F_{ij}$.

Under more restrictive assumptions (corresponding informally to uniform controls instead of local ones), the rate (with respect to $\theta$) of $C(\theta,\Psi,\mu_{H},\nu_{F})$ given in Theorem \ref{thm:counting:process:convergence} is linear in comparison with a rate which is at least exponential in general. The main assumption corresponds to the stability criterion of Hawkes processes \cite{Bre_Massou}.
\begin{prop}\label{prop:linear:bound}
Under (\hyperref[ass:mu:H:locally:square:integrable]{$\mathcal{A}^{\mu_{H}}_{\infty,2}$}), (\hyperref[ass:nu:F:variance]{$\mathcal{A}^{\nu_{F}}_{2}$}) and (\hyperref[ass:Psi:Lipschitz]{$\mathcal{A}^{\Psi}_{\rm Lip}$}), assume either \hyperref[ass:H1]{$\mathcal{H}_{1}$} or \hyperref[ass:H2]{$\mathcal{H}_{2}$}. Furthermore, assume that both:
\begin{itemize}
\nsp\nsp
\item the functions $M_{\mu_{H}}$ and $\Psi$ are such that $\alpha:={\rm Lip}(\Psi) || M_{\mu_{H}} ||_{1}<1$ and $|| M_{\mu_{H}} ||_{2} <\infty$;
\nsp
\item the functions $m_{\nu_{F}}$ and $V_{\nu_{F}}$ are uniformly bounded.
\end{itemize}
\nsp
Then, the constant $C(\theta,\Psi,\mu_{H},\nu_{F})$ given in Theorem~\ref{thm:counting:process:convergence} can be bounded by $\beta(\Psi,\mu_{H},\nu_{F}) \, \theta$ where the constant $\beta(\Psi,\mu_{H},\nu_{F})$ depends neither on $\theta$ nor on $n$. 

This bound holds for all $\theta\geq 0$ under \hyperref[ass:H2]{$\mathcal{H}_{2}$} whereas it holds for $\theta<(1-\alpha)/|| \Psi ||_{\infty}$ under \hyperref[ass:H1]{$\mathcal{H}_{1}$}. 
\end{prop}

A proof is given in  \ref{sec:proof:prop:linear:bound} where an explicit expression of $\beta$ can be found in Equation \eqref{eq:explicit:bound:beta:delta>0} or \eqref{eq:explicit:bound:beta} depending on the context.\\

As said in the introduction, Hawkes processes seem to be the microscopic point processes underpinning the \eqref{eq:PPS:intro} system introduced in \cite{pakdaman2010dynamics}. There is a striking similarity, modulo a change of time, between :
\begin{itemize}
\item on the one hand, the mean intensity, denoted by $m(t)$, of a linear Hawkes process which is a linear function of $\int_{0}^{t} h(t-z)m(z)dz$ (this is not true in the nonlinear case since a nonlinear intensity function does not commute with the expectation),
\item and on the other hand, the firing rate $p$ in \eqref{eq:PPS:intro} which is a function of $\int_{0}^{t} d(z) n(0,t-z)dz$. 
\end{itemize}
A first step in this direction has been made in \cite{chevallier2015microscopic} in the framework of a network of  i.i.d. Hawkes processes. 
In that case, there is no direct bridge between Hawkes processes (even linear Hawkes processes) and the \eqref{eq:PPS:intro} system as it is shown in \cite{chevallier2015microscopic}. Indeed, when the size of the network goes to infinity, one recovers conditional expectation of the intensity with respect to the age (instead of the mean intensity). By comparison, in the mean-field framework of the present article, the averaging phenomenon takes place at the level of the point measures as it can be seen in Equation \eqref{eq:limit:equation:counting:process} (there is no need to commute the intensity function with the expectation). Equation \eqref{eq:mckean:vlasov:intensity} indeed gives the limit intensity $\overline{\lambda}(t)$ as a function of $\int_{0}^{t} h(t-z)\overline{\lambda}(z)dz$. Furthermore, this limit intensity is given by the macroscopic system.

System \eqref{eq:definition:coupling:ADRHP}-\eqref{eq:definition:coupling:limit:equation} provides an efficient coupling between the spikes attached with the $n$-particle system and the spikes associated with the limiting process. In order to go one step further, a natural question is to wonder about a possible coupling between the ages associated with the two dynamics.
This question is not addressed in \cite{delattre2016} in which the propagation of chaos is discussed at the level of the counting processes only. In comparison, we are here willing to investigate this question carefully. The underlying motivation is not of a mathematical essence only: exhibiting a suitable coupling between the ages of the ADRHP and the ages of the point processes of the McKean-Vlasov type whose intensity depends on time and on the age is the right and proper way to make the connection between the microscopic description of the neural dynamics and the macroscopic equation \eqref{eq:PPS:intro}.\\

In the sequel, Assumption (\hyperref[ass:initial:condition:bounded]{$\mathcal{A}^{\zeta_{N_{-}}}_{\infty}$}) is used. 
It appears that the dependence of the age at time $0$ with respect to the initial condition generates additional difficulties for investigating the mean-field approximation.
To limit the complexity of the analysis, it is quite convenient to assume that the age at time $0$ is bounded, which is precisely what Assumption (\hyperref[ass:initial:condition:bounded]{$\mathcal{A}^{\zeta_{N_{-}}}_{\infty}$}) says.

\begin{cor}\label{cor:age:process:convergence}
With the notations and assumptions of Theorem \ref{thm:counting:process:convergence}, assume that (\hyperref[ass:initial:condition:bounded]{$\mathcal{A}^{\zeta_{N_{-}}}_{\infty}$}) is satisfied.

Then, the age processes $(S^{n,i}_{t-})^{i=1,\dots,n}_{t\geq 0}$ associated with the sequence of ADRHP $(N^{n,i})_{i=1,..,n}$ and the age processes $(\overline{S}^{i}_{t-})^{i\geq 1}_{t\geq 0}$ associated with the i.i.d. solutions $(\overline{N}^i_t)^{i\geq 1}_{t\geq 0}$ of the limit equation satisfy for all $i=1,\dots ,n$ and $\theta>0$,
\begin{equation}\label{eq:wasserstein:age:process}
\mathbb{E}\left[ \sup_{t\in [0,\theta]} | S^{n,i}_{t-} -\overline{S}^i_{t-} |  \right] \leq (M_{T_{0}}+\theta) \,  C(\theta,\Psi,\mu_{H},\nu_{F}) \, n^{-1/2},
\end{equation}
where $C(\theta,\Psi,\mu_{H},\nu_{F})$ is  given in Theorem \ref{thm:counting:process:convergence} and $M_{T_{0}}$ is defined in (\hyperref[ass:initial:condition:bounded]{$\mathcal{A}^{\zeta_{N_{-}}}_{\infty}$}).
\end{cor}
\begin{rem}\label{rem:double:age:process}
Reminding the strong connection between the predictable age process and the standard one, stated below Equation \eqref{eq:def:age:process:predictable}, it is clear that the bound \eqref{eq:wasserstein:age:process} is also valid when replacing the predictable age processes by their standard counterparts.
\end{rem}

\begin{proof}
Let us note that, for all $n\geq 1$ and $i=1,\dots ,n$, $N^{n,i}$ and $\overline{N}^{i}$ coincide on the non-positive part, i.e. $N^{n,i}_{-}=\overline{N}^{i}_{-}$. Therefore, $S^{n,i}_{0}=\overline{S}^{i}_{0}$ and  $ \sup_{t\in [0,\theta]} | S^{n,i}_{t-} -\overline{S}^i_{t-} | $ is a.s. upper bounded by $M_{T_{0}}+\theta$ when the trajectories $(S^{n,i}_{t-})_{t\in[0,\theta]}$ and $(\overline{S}^{i}_{t-})_{t\in [0,\theta]}$ are different and is equal to 0 otherwise. Therefore, we have the following bound
\begin{equation*}
\mathbb{E}\left[  \sup_{t\in [0,\theta]} | S^{n,i}_{t-} -\overline{S}^i_{t-} |  \right]\leq (M_{T_{0}}+\theta) \mathbb{P}\left( \left(S_{t-}^{n,i}\right)_{t\in\left[0,\theta\right]}\neq\left(\overline{S}_{t-}^{i}\right)_{t\in\left[0,\theta\right]} \right).
\end{equation*}
Yet if the trajectories are different, there is at least one point between $0$ and $\theta$ which is not common to both $N^{n,i}_{+}$ and $\overline{N}^{i}_{+}$, that is $\sup_{t\in [0,\theta]} |N^{n,i}_{t} - \overline{N}^{i}_{t}| \neq 0$, hence
\begin{equation}\label{eq:age:different:bound}
\mathbb{P}\left( \left(S_{t-}^{n,i}\right)_{t\in\left[0,\theta\right]}\neq\left(\overline{S}_{t-}^{i}\right)_{t\in\left[0,\theta\right]} \right) = \mathbb{P}\left( \sup_{t\in [0,\theta]} |N^{n,i}_{t} - \overline{N}^{i}_{t}| \neq 0 \right).
\end{equation}
Moreover, since counting processes are piecewise constant with jumps of height $1$ a.s., it is clear that
\begin{equation}\label{eq:more:than:one:point:bound}
\mathbb{P}\Big( \sup_{t\in [0,\theta]} |N^{n,1}_{t} - \overline{N}^{1}_{t}| \neq 0 \Big)  \leq \mathbb{E}\Big[ \sup_{t\in [0,\theta]} | N^{n,1}_{t} -\overline{N}^1_{t} | \Big],
\end{equation}
where we used Markov's inequality.
Finally, inequality~\eqref{eq:wasserstein:age:process} clearly follows from Theorem~\ref{thm:counting:process:convergence}.
\end{proof}

\subsection{Mean-field approximations}
\label{sec:mean:field:counting}
Inspired by the seminal work of Sznitman \cite{Sznitman_91}, we now obtain, from the results of the previous section, the convergence of the $n$-particle system towards the limit equation:
\begin{itemize}
\item the empirical distribution of the point processes associated with the $n$-particle system converges to the distribution of the point process solution of the limit equation,
\item the empirical distribution of the age processes associated with the $n$-particle system converges to the distribution of the age process associated with the solution of the limit equation,
\end{itemize}
This result, together with the ones from the previous paragraphs, are typical of what is known as the \emph{propagation of chaos theory of interacting particle system}. In particular, it says that $k$ fixed neurons behave independently and identically when the size of the network goes to infinity. Their spiking dynamics being described by the limit equation \eqref{eq:limit:equation:counting:process}.

\begin{cor}\label{cor:mean:field:approx:counting:and:age:process}
Let $\mathcal{L}(X)$ denote the distribution of some random variable $X$ and $\mathcal{D}(\mathbb{R}_{+})$ denote the space of c\`adl\`ag functions from $\mathbb{R}_{+}$ to $\mathbb{R}$ endowed with the Skorokhod topology.

With notations and assumptions of Theorem~\ref{thm:counting:process:convergence}, we have the following mean-field approximations:
\begin{itemize}
\item the weak convergence in $\mathcal{P}(\mathcal{D}(\mathbb{R}_{+}))$ of the empirical measure of counting processes,
\begin{equation}\label{eq:mean:field:counting}
\frac{1}{n}\sum_{i = 1}^n \delta_{(N^{n,i}_t)_{t\geq 0}}\xrightarrow[n\rightarrow \infty]{} \mathcal{L}\big(( \overline{N}^{1}_t)_{t\geq 0}\big);
\end{equation}
\item if furthermore (\hyperref[ass:initial:condition:bounded]{$\mathcal{A}^{\zeta_{N_{-}}}_{\infty}$}) holds, the weak convergence in $\mathcal{P}(\mathcal{D}(\mathbb{R}_{+}))$ of the empirical measure of the standard age processes,
\begin{equation}\label{eq:mean:field:age}
\frac{1}{n}\sum_{i = 1}^n \delta_{(S^{n,i}_{t})_{t\geq 0}}\xrightarrow[n\rightarrow \infty]{} \mathcal{L}\big(( \overline{S}^{1}_{t})_{t\geq 0}\big).
\end{equation}
\end{itemize}
Both convergences also hold in probability since the limits are constant in $\mathcal{P}(\mathcal{D}(\mathbb{R}_{+}))$.

\begin{sloppypar}
Finally, if furthermore (\hyperref[ass:Psi:uniformly:bounded]{$\mathcal{A}^{\Psi}_{\infty}$}) and (\hyperref[ass:initial:condition:density]{$\mathcal{A}^{\zeta_{N_{-}}}_{u^{\rm in}}$}) hold, then the unique solution $u$ of the system \eqref{eq:edp:PPS} with initial condition that  $u(0,\cdot)=u^{\rm in}$ is such that $u_{t}:=u(t,\cdot)$ is the density of the age $\overline{S}^{1}_{t-}$ and for all $\theta>0$, 
\begin{equation}\label{eq:rate:Wasserstein}
\sup_{t\in [0,\theta]} \mathbb{E}\left[ W_{1}\left( \frac{1}{n}\sum_{i = 1}^n \delta_{S^{n,i}_{t-}} , u_{t} \right) \right]  \leq   
D(\theta,\Psi,\mu_{H},\nu_{F},M_{T_{0}}) \, n^{-1/2},
\end{equation}
where $W_{1}$ denotes the standard $1$-Wasserstein distance and the constant $D(\theta,\Psi,\mu_{H},\nu_{F},M_{T_{0}})$ does not depend on $n$.
\end{sloppypar}
\end{cor}
\begin{rem}
Of course, the convergence \eqref{eq:mean:field:age} is also valid when replacing the standard age processes by their predictable counterparts. However, let us mention that the predictable age processes belong to $\mathcal{G}(\mathbb{R}_{+})$, the space of càglàd functions (continuous to the left with right limits). Hence the convergence of the empirical measure of the predictable age processes holds in $\mathcal{P}(\mathcal{G}(\mathbb{R}_{+}))$, where we endow $\mathcal{G}(\mathbb{R}_{+})$ with an analoguous of the Skorokhod topology
\end{rem}
\begin{proof}
The space of càdlàg functions $\mathcal{D}(\mathbb{R}_{+})$ endowed with the Skorokhod topology is a Polish space. 
So, according to \cite[Proposition 2.2]{Sznitman_91} or \cite[Proposition 4.2]{Meleard_96}, to show the first limit \eqref{eq:mean:field:counting}, it suffices to check that $((N^{n,1}_t)_{t\geq 0},(N^{n,2}_t)_{t\geq 0})$ converges in distribution, as $n\rightarrow +\infty$, to two independent copies of $(\overline{N}^{1}_t)_{t\geq 0}$. Since the convergence with respect to Lipschitz continuous test functions is sufficient in order to prove the convergence in distribution (Portemanteau Theorem \cite{Klenke2007}), the first limit clearly follows from both \eqref{eq:bound:full:expect:counting:process} and the fact that the uniform convergence topology on compact time intervals is finer than the Skorokhod topology.
The proof of the second limit is similar with the difference that it follows from \eqref{eq:wasserstein:age:process} instead of \eqref{eq:bound:full:expect:counting:process}.

The link between the solution of \eqref{eq:edp:PPS} and the age processes associated with the solution of the limit equation \eqref{eq:limit:equation:counting:process} is given by Proposition \ref{prop:age:process:limit:equation:PPS}. 

The rate of convergence for the $1$-Wasserstein distance stated in \eqref{eq:rate:Wasserstein} is a consequence of the rate of convergence for i.i.d. real valued random variables. Fix $\theta>0$ and let $t$ be in $[0,\theta]$. First, using the exchangeability of the particles, it follows from Corollary \ref{cor:age:process:convergence} that there exists a constant $C(\theta,\Psi,\mu_{H},\nu_{F},M_{T_{0}})$ such that
\begin{equation*}
\mathbb{E}\left[ W_{1}\left( \frac{1}{n}\sum_{i = 1}^n \delta_{S^{n,i}_{t-}} , \frac{1}{n}\sum_{i = 1}^n \delta_{\overline{S}^{i}_{t-}} \right) \right] \leq   C(\theta,\Psi,\mu_{H},\nu_{F},M_{T_{0}}) \, n^{-1/2}.
\end{equation*} 
Then, applying \cite[Theorem 1]{fournier2014rate} to the i.i.d. random variables $\overline{S}^{i}_{t-}$, that are bounded by $M_{T_{0}}+ \theta$, we deduce that there exists a constant $\tilde{C}(\theta,M_{T_{0}})$ such that $\mathbb{E}[ W_{1}( \frac{1}{n}\sum_{i = 1}^n \delta_{\overline{S}^{i}_{t-}} , u_{t} ) ]\leq \tilde{C}(M_{T_{0}},\theta) n^{-1/2}$.
Finally, the triangular inequality for the $1$-Wasserstein distance gives \eqref{eq:rate:Wasserstein}.
\end{proof}

The first mean-field approximation \eqref{eq:mean:field:counting} is a generalization of the one given in \cite[Theorem 8-{\it (iii)}]{delattre2016} where mean-field interacting Hawkes processes are approximated by Poisson processes of the McKean-Vlasov type. Here, the limit processes are point processes of the McKean-Vlasov type whose intensity depends on time (like Poisson processes) and on the age.

Moreover, Equation \eqref{eq:rate:Wasserstein} extends the result on the rate of convergence for the age processes given in \cite[Section 5]{quininao2015microscopic}.

\section{Conclusion}

We present a generalization of mean-field interacting Hawkes processes, namely age dependent random Hawkes processes (ADRHPs), which are well-adapted to neuroscience modelling. From a biological point of view, they encompass some interesting features such as refractory period, synaptic integration or random synaptic weights. These processes are studied in a mean-field situation and we show in Theorem \ref{thm:counting:process:convergence} and Corollary \ref{cor:mean:field:approx:counting:and:age:process} that, as the number of particles goes to infinity, they can be well approximated by point processes of the McKean-Vlasov type whose intensity depends on time and on the age. These limit point processes are closely related to the age structured PDE system introduced by Pakdaman, Perthame and Salort, namely \eqref{eq:PPS:intro}, as shown in Proposition \ref{prop:age:process:limit:equation:PPS}.

Hence, using the theory of mean-field approximations, the present article makes a bridge between the microscopic modelling given by Hawkes processes, or more generally age dependent random Hawkes processes, and the macroscopic modelling given by the \eqref{eq:PPS:intro} system. This bridge is presented under the main assumption that the intensity of the microscopic point processes is bounded. In this sense, the present article offers an answer to the question left open in \cite{chevallier2015microscopic}. This legitimises the convolution term $X(t)$ in the \eqref{eq:PPS:intro} system as well as opens the way to the study of new assumptions on the spiking rate $p$ appearing in the \eqref{eq:PPS:intro} system from a more analytical point of view. Up to our knowledge, this has not been done yet.

The present article gives somehow the law of large numbers for a generalization of Hawkes processes. It could be interesting to investigate how these processes fluctuates around their mean limit or in other words find some kind of functional central limit theorem for Hawkes processes in a mean-field framework.

As noted, random synaptic weights can be considered in this study. However, they are supposed to be, in some sense, independent and identically distributed which can be considered as an unrealistic assumption. Inspired by \cite{faugeras2014asymptotic}, it could be interesting to see how correlated synaptic weights could be handled in the Hawkes processes framework.

On a different path, it could be interesting to see how locally stationary Hawkes processes, as introduced in \cite{rouefftime}, behave in a mean-field situation. Indeed, these processes may take into account the dynamics of the synaptic weights occurring in the neural network.

\paragraph{Acknowledgement}
This research was partly supported by the french Agence Nationale de la Recherche (ANR 2011 BS01 010 01 projet Calibration) and by the interdisciplanary axis MTC-NSC of the University of Nice Sophia-Antipolis.
I would like to thank Patricia Reynaud-Bouret and François Delarue for helpful discussions which improved this paper.

\pagebreak
\appendix

\section{Proofs}

\subsection{Proof of Proposition \ref{prop:Thinning:well:posed}}
\label{sec:proof:prop:Thinning:well:posed}

Let us denote $G_{i}:t\mapsto  \sup_{s\geq 0} \Psi(s,0) + {\rm Lip}(\Psi) n^{-1} \sum_{j=1}^{n} | F_{ij}(t) |$. Thanks to (\hyperref[ass:nu:F:locally:bounded]{$\mathcal{A}^{\nu_{F}}_{1}$}) and Fubini's Theorem we have, for all $T>0$, $\mathbb{E}[\int_{0}^{T} |F_{11}(t)| dt ]= \int_{0}^{T} \mathbb{E}[|F_{11}(t)|] dt<+\infty$. In particular, for all $1\leq i,j\leq n$, $t\mapsto | F_{ij}(t) |$ is locally integrable almost surely. Hence, there exists a subset $\Omega$ of probability 1 such that, on $\Omega$, $G_{i}$ is locally integrable for all $i$. Fixing the $G_{i}$'s, one can apply Lemma \ref{lem:stochastic:domination} (with $a_{i}={\rm Lip}(\Psi)$ and $g_{i}=G_{i}$) to deduce that the processes $(N_t^i)^{i=1,..,n}_{t \geq 0}$ are dominated by the processes $(\tilde{N}_t^i)^{i=1,..,n}_{t \geq 0}$ (defined by \eqref{eq:Hawkes:Thinning:dominating}) and so are  well-defined.

It remains to show that the function $t\mapsto \mathbb{E}\left[ N_{t}^{1} \right]$ is locally bounded. First, let us study the dominating processes. We have
\begin{equation*}
\mathbb{E}\left[ \frac{1}{n} \sum_{i=1}^{n} \tilde{N}^{i}_{t} \right]  \leq   \frac{1}{n} \sum_{i=1}^{n} \int_{0}^{t}  \mathbb{E}\left[G_{i}(t')\right]  dt' 
+ {\rm Lip}(\Psi)  \int_{0}^{t} M_{\mu_{H}}(t-t') \mathbb{E}\left[ \frac{1}{n}\sum_{j=1}^{n} \tilde{N}^{j}_{t'} \right] dt',
\end{equation*}
where we used Lemma \ref{lem:double:Fubini}. Next, $t\mapsto \mathbb{E}\left[G_{i}(t)\right]=\sup_{s\geq 0} \Psi(s,0) + {\rm Lip}(\Psi) \mathbb{E}\left[ | F_{11}(t) | \right]$ is locally bounded and $M_{\mu_{H}}$ is locally integrable so Lemma \ref{lem:Gronwall:Picard}-{\it (i)} gives that $$t\mapsto \mathbb{E}\left[ n^{-1} \sum_{i=1}^{n} \tilde{N}^{i}_{t} \right]=\mathbb{E}\left[ \tilde{N}^{1}_{t} \right]$$ is locally bounded. Finally, the stochastic domination (in particular, $\mathbb{E}[ N_{t}^{1}  ]\leq \mathbb{E}[ \tilde{N}_{t}^{1} ]$) gives the result.

\subsection{Proof of Proposition \ref{prop:existence:uniqueness:PPS:linear}}
\label{sec:proof:prop:existence:uniqueness:PPS:linear}

First, in order to be consistent with the formalism used in \cite{canizo2013measure} we must rewrite the system \eqref{eq:edp:PPS:linear} in a single equation in the following way
\begin{equation}\label{eq:PPS:linear:CCC}
\frac{\partial u\left(t,s\right)}{\partial t}+\frac{\partial u\left(t,s\right)}{\partial s} = N_{1}(t,s,u)+ \delta_{s=0}n_{2}(t,u),
\end{equation}
with initial condition $u^{\rm in}$ where
\begin{equation*}
\begin{cases}
N_{1}(t,s,u):= -f(t,s)u(t,s)\\ 
n_{2}(t,u):= \int_{0}^{+\infty} f(t,s') u(t,s') ds'.
\end{cases}
\end{equation*}
The use of the Dirac mass localized in age equal to $0$ represents the boundary condition that is the second equation of \eqref{eq:edp:PPS:linear}. 

Note that the general result \cite[Theorem 2.4.]{canizo2013measure} gives existence and uniqueness of solution in $BC(\mathbb{R}_{+},\mathcal{M}(\mathbb{R}))$ and not $BC(\mathbb{R}_{+},\mathcal{M}(\mathbb{R}_{+}))$ even if the initial condition has support contained in $\mathbb{R}_{+}$. However, as explained in \cite[Section 3.3.]{canizo2013measure}, it suffices to extend the equation for $s$ in $\mathbb{R}$, to apply \cite[Theorem 2.4.]{canizo2013measure} and then to check that the support of the solution is conserved in the sense that: if $u^{\rm in}$ has support on $\mathbb{R}_{+}$ then the unique solution given by the Theorem has also support contained in $\mathbb{R}_{+}$ for all time $t\geq 0$.

Hence, consider Equation \eqref{eq:PPS:linear:CCC} but with $s$ being in $\mathbb{R}$ by mirror symmetry for definiteness (that is $N_{1}(t,s,u)=N_{1}(t,-s,u)$ and so $f(t,s)=f(t,-s)$). Let us check that the assumptions of \cite[Theorem 2.4.]{canizo2013measure} are satisfied.

-(H1) and (H2) are clearly satisfied.

-(H3). We need to verify that $N_{1}$ and $n_{2}$ are continuous in $t$ with respect to the usual topology and in $u$ with respect to the topology induced by the bounded Lipschitz norm (denoted by $||.||_{BL}$). On the one hand, using the boundedness of $f$, we have, with $u_{t}$ denoting the measure $u(t,\cdot)$,
\begin{multline*}
|| N_{1}(t+t',s,u+\tilde{u}) - N_{1}(t,s,u) ||_{BL} \leq \sup_{s\in \mathbb{R}} |f(t+t',s)-f(t,s)| || u_{t} ||_{BL}\\
+ || f ||_{\infty} || \tilde{u}_{t+t'}||_{BL} + || f ||_{\infty} ||u_{t+t'}-u_{t}||_{BL}.
\end{multline*}
As $t'$ and $\tilde{u}$ converge to $0$, the first term converges to $0$ since $f$ is (uniformly in $s$) continuous with respect to $t$, the second one clearly converges to $0$ and the third one converges to $0$ since $u$ belongs to $\mathcal{BC}(\mathbb{R}_{+},\mathcal{M}(\mathbb{R}))$.

On the other hand, using once again the boundedness of $f$, we have,
\begin{equation*}
| n_{2}(t+t',u+\tilde{u}) - n_{2}(t,u) | \leq \sup_{s\in \mathbb{R}} |f(t+t',s)-f(t,s)|  || u_{t} ||_{BL}
 + || f ||_{\infty} ||u_{t+t'}-u_{t}||_{BL},
\end{equation*}
which converges to $0$ as $t'$ tends to $0$.

-(H4). It suffices to show that $N_{1}$ and $n_{2}$ are Lipschitz continuous with respect to the variable $u$. On the one hand, we have
\begin{equation*}
|| N_{1}(t,s,u) - N_{1}(t,s,v) ||_{BL} \leq || f ||_{\infty} ||u_{t}-v_{t}||_{BL}.
\end{equation*}
On the other hand, we have
\begin{equation*}
| n_{2}(t,u) - n_{2}(t,v) |\leq || f ||_{\infty} ||u_{t}-v_{t}||_{BL}.
\end{equation*}

-(H5). Here, it suffices to check that $N_{1}(t,s,u)+ \delta_{s=0}n_{2}(t,u)$ carries bounded sets in total variation norm to bounded sets in total variation norm. Denoting $||.||_{TV}$ the total variation norm, we have $||N_{1}(t,s,u) + \delta_{s=0}n_{2}(t,u)||_{TV}\leq 2 ||f||_{\infty} || u||_{TV}$.

Finally, the argument to prove conservation of the support for solutions being the same as the one elaborated in \cite[Section 3.3.]{canizo2013measure}, it is not reproduced here.

\subsection{Proof of Proposition \ref{prop:characteristics:PPS:linear}}
\label{sec:proof:prop:characteristics:PPS:linear}
 
The method of characteristics applied to the first equation of \eqref{eq:edp:PPS:linear} suggests to consider:
\begin{itemize}
\item for all $z\geq 0$, $u^{z}:t\mapsto u(t,t+z)$ satisfying
\begin{equation*}
\frac{d}{dt} u^{z}(t) = -f(t,z+t)u^{z}(t)
\end{equation*}
hence for all $t\geq 0$,
\begin{equation*}
u^{z}(t)=u^{z}(0) \exp\left( -\int_{0}^{t} f(t',z+t') dt' \right)
\end{equation*}
and so, using that $u^{z}(0)=u^{\rm in}(z)$ and letting $s=z+t$, one has \eqref{eq:characteristics:s>t}.
\item for all $z\geq 0$, $u^{z}:s\mapsto u(s+z,s)$ satisfying
\begin{equation*}
\frac{d}{ds} u^{z}(s) = -f(s+z,s) u^{z}(s)
\end{equation*}
hence for all $s\geq 0$,
\begin{equation*}
u^{z}(s)= u^{z}(0) \exp\left( -\int_{0}^{s} f(s'+z,s') ds' \right)
\end{equation*}
and so, using that $u^{z}(0)=u(z,0)$ and letting $t=z+s$, one has
\begin{equation*}
u(t,s) = u(t-s,0) \exp\left( -\int_{0}^{s} f(t-s+s',s') ds'\right),\quad \text{for $t\geq s$,}
\end{equation*}
which is not exactly \eqref{eq:characteristics:t>s}. 
Here, $u(t-s,0)$ is just a parameter which is not constrained by the first equation of \eqref{eq:edp:PPS:linear}. However, it is characterized by the second equation of \eqref{eq:edp:PPS:linear} as explained below. 
\end{itemize}
For any $T>0$, consider the map
\begin{equation*}
\begin{array}{cccc}
G_{T}:\ & L^{\infty}([0,T]) &\longrightarrow& L^{\infty}([0,T])\\
 & u_{0} &\longmapsto &\left(\displaystyle t\mapsto G(u_{0})(t) \right),
\end{array}
\end{equation*}
where 
\begin{multline*}
G(u_{0})(t):=  \int_{0}^{t} f(t,s) u_{0}(t-s) \exp\left( -\int_{0}^{s} f(t-s+s',s') ds'\right) ds\\
+ \int_{t}^{+\infty} f(t,s) u^{\rm in}(s-t) \exp\left( -\int_{0}^{t} f(t',s-t+t') dt'\right) ds.
\end{multline*}
Note that the characteristics and the second equation of \eqref{eq:edp:PPS:linear} suggest that, denoting $u$ the solution of \eqref{eq:edp:PPS:linear}, $u(\cdot ,0)$ is bounded and that its trace on $[0,T]$ is a fixed point of $G_{T}$.  Using the boundedness of $f$ and the fact that the argument in the exponential is non-positive, we have for any $u_{0},v_{0}$ in $L^{\infty}([0,T])$,
\begin{equation*}
\left| G_{T}(u_{0})(t) -G_{T}(v_{0})(t) \right| \leq  ||f||_{\infty} \int_{0}^{t} |u_{0}(t-s) - v_{0}(t-s)| ds
\leq  T ||f||_{\infty} ||u_{0}-v_{0}||_{ L^{\infty}([0,T])}. 
\end{equation*}
Now, fix $T>0$ such that $T|| f ||_{\infty}\leq 1/2$ so that $G_{T}$ is a contraction and admits a unique fixed point. Note that $G_{T}$ maps non-negative functions to non-negative functions, so that the fixed point of $G_{T}$ is a non-negative function.
Iterating this fixed point gives the existence and uniqueness of a locally bounded function $u_{0}$ (which is non-negative) such that for all $t\geq 0$, $u_{0}(t)=G(u_{0})(t)$ (see the end of the proof of Lemma \ref{lem:Xu:well:defined} for the same kind of argument in a more detailed form). 
Until the end of the proof, $u_{0}$ will denote this fixed point.

It only remains to check that $u$ is a solution of \eqref{eq:edp:PPS:linear} in the weak sense.
Let $\varphi$ be in $\mathcal{C}^{\infty}_{c,b}(\mathbb{R}_{+}^{2})$, let us compute $\int_{\mathbb{R}_{+}^{2}} \left(\frac{\partial}{\partial t}+\frac{\partial}{\partial s}\right) \varphi\left(t,s\right)u\left( t, s\right) dt ds$. Sticking with the decomposition of the representation given by \eqref{eq:characteristics:s>t}-\eqref{eq:characteristics:t>s} and using integration by parts to go backward in the heuristic given by the method of characteristics, one has
\begin{multline}\label{eq:weak:sense:s>t}
\int_{s\geq t} \left(\frac{\partial}{\partial t}+\frac{\partial}{\partial s}\right) \varphi\left(t,s\right)u\left( t, s\right) dt ds = \int_{s\geq t} \varphi(t,s) f(t,s) u(t,s) dt ds \\
- \int_{s=0}^{+\infty} \varphi(0,s) u^{\rm in}(s) ds,
\end{multline}
and
\begin{multline}\label{eq:weak:sense:t>s}
\int_{t\geq s} \left(\frac{\partial}{\partial t}+\frac{\partial}{\partial s}\right) \varphi\left(t,s\right)u\left( t, s\right) dt ds = \int_{t\geq s} \varphi(t,s) f(t,s) u(t,s) dt ds \\
- \int_{t=0}^{+\infty} \varphi(t,0) u_{0}(t) dt.
\end{multline}
Remarking that the definition of $u_{0}$ as the fixed point of $G$ gives that $u$ defined by \eqref{eq:characteristics:s>t}-\eqref{eq:characteristics:t>s} satisfies the second equation of \eqref{eq:edp:PPS:linear} in a strong sense, one deduces that \eqref{eq:edp:PPS:linear:weak:sense} is satisfied by gathering \eqref{eq:weak:sense:s>t} and \eqref{eq:weak:sense:t>s}.

\subsection{Proof of \eqref{eq:F:contraction}}
\label{sec:proof:thm:existence:uniqueness:PPS}
Let $Y_{1}$ and $Y_{2}$ be two fixed locally bounded functions and denote $u_{1}:=u_{Y_{1}}$ and $u_{2}:=u_{Y_{2}}$ the two solutions associated with $Y_{1}$ and $Y_{2}$ (with same initial condition $u^{\rm in}$) and $\tilde{u}(t):=u_{1}(t,0) - u_{2}(t,0)$ for all $t\geq 0$. Using the characteristics, i.e. \eqref{eq:characteristics:s>t} and \eqref{eq:characteristics:t>s}, one deduces from the second equation in \eqref{eq:edp:PPS:frozen} that for $i=1$ or $2$,
$
u_{i}(t,0)= \upsilon_{i}^{0,t} + \upsilon_{i}^{t,\infty}
$
where
\begin{equation*}
\left\{
\begin{aligned}
& \upsilon_{i}^{0,t}:= \int_{0}^{t} \Psi (s,Y_{i}(t)+f_{0}(t)) u_{i}(t-s,0) e^{-\int_{0}^{s} \Psi(s', Y_{i}(t-s+s') +f_{0}(t-s+s')) ds'} ds\\
& \upsilon_{i}^{t,\infty} :=\int_{t}^{+\infty} \Psi (s,Y_{i}(t)+f_{0}(t)) u^{\rm in}(s-t) e^{ -\int_{0}^{t} \Psi(s-t+t', Y_{i}(t') +f_{0}(t')) dt'}ds.
\end{aligned}
\right.
\end{equation*}
and so $\tilde{u}(t)=A(t)+B(t)$ where $A(t):=  \upsilon_{1}^{0,t} - \upsilon_{2}^{0,t}$ and $B(t):= \upsilon_{1}^{t,\infty} - \upsilon_{2}^{t,\infty}$. Before studying the functions $A$ and $B$, let us remark that in order to prove \eqref{eq:F:contraction}, it suffices to prove that there exists a non-decreasing function $C$ (which depends neither on the initial condition nor on $f_{0}$) such that for all $T\geq 0$,
\begin{equation}\label{eq:control:A:B}
|| \tilde{u} ||_{L^{\infty}([0,T])} \leq C(T) || Y_{1}-Y_{2} ||_{L^{\infty}([0,T])}.
\end{equation}
Indeed, using the definition of $F_{T}$ given by \eqref{eq:def:FT}, one then deduces that for all $T\geq 0$,
\begin{equation*}
|| F_{T}(Y_{1}) - F_{T}(Y_{2}) ||_{L^{\infty}([0,T])} \leq C(T) || Y_{1}-Y_{2} ||_{L^{\infty}([0,T])} \int_{0}^{T} | h(z) | dz
\end{equation*}
and since $h$ is locally integrable and $C$ is non-decreasing, there exists $T>0$ small enough such that $C(T) \int_{0}^{T} | h(z) | dz\leq 1/2$ which ends the proof. To prove \eqref{eq:control:A:B}, let $t$ be a positive real number.
\paragraph{Study of $A$.}
We have $A=A_{1}+A_{2}+A_{3}$ with
\begin{equation*}
\begin{cases}
A_{1}(t):= \int_{0}^{t} [\Psi (s,Y_{1}(t)+f_{0}(t))- \Psi (s,Y_{2}(t)+f_{0}(t))] u_{1}(t-s,0)
 e^{ -\int_{0}^{s} \Psi(s', Y_{1}+f_{0}) ds' } ds \\
A_{2}(t):= \int_{0}^{t} \Psi (s,Y_{2}(t)+f_{0}(t)) [u_{1}(t-s,0) - u_{2}(t-s,0)]  e^{ -\int_{0}^{s} \Psi(s', Y_{1} +f_{0}) ds' } ds \\
A_{3}(t):= \int_{0}^{t} \Psi (s,Y_{2}(t)+f_{0}(t))  u_{2}(t-s,0) [e^{ -\int_{0}^{s} \Psi(s', Y_{1} +f_{0}) ds'} - e^{ -\int_{0}^{s} \Psi(s', Y_{2} +f_{0}) ds'}]  ds.
\end{cases}
\end{equation*}
where the arguments ``$t-s+s'$'' in the exponentials are not written for simplicity.

- \underline{Study of $A_{1}$.}
Using the Lipschitz continuity of $\Psi$, the fact that the argument in the exponential is non-positive and the a priori bound on $u_{1}$, we have
\begin{multline}\label{eq:control:A1}
|A_{1}(t)|\leq {\rm Lip}(\Psi) \int_{0}^{t} u_{1}(t-s,0) |Y_{1}(t) -Y_{2}(t) | ds \\
\leq {\rm Lip}(\Psi) \max(M,|| \Psi ||_{\infty}) \, t\,  || Y_{1}-Y_{2} ||_{L^{\infty}([0,t])}.
\end{multline}

- \underline{Study of $A_{2}$.}
Using the boundedness of $\Psi$ and the fact that the argument in the exponential is non-positive, we have
\begin{equation}\label{eq:control:A2}
|A_{2}(t)|\leq || \Psi ||_{\infty} \int_{0}^{t} |u_{1}(z,0) - u_{2}(z,0)|dz = || \Psi ||_{\infty} \int_{0}^{t} |\tilde{u}(z,0)|dz.
\end{equation}

- \underline{Study of $A_{3}$.}
The arguments of the exponentials are non-positive and the exponential function is Lipschitz with constant 1 on $\mathbb{R}_{-}$ so, using the a priori bound on $u_{2}$, we have
\begin{eqnarray}
|A_{3}(t)| &\leq & || \Psi ||_{\infty} || u_{2} ||_{\infty} \! \int_{0}^{t} \left| \int_{0}^{s} \Psi(s', Y_{1} +f_{0}) ds' - \int_{0}^{s} \Psi(s', Y_{2} +f_{0}) ds' \right|  ds \nonumber\\
  & \leq & || \Psi ||_{\infty} \max(M,|| \Psi ||_{\infty}) {\rm Lip}(\Psi) \, t^{2}\,  || Y_{1}-Y_{2} ||_{L^{\infty}([0,t])}. \label{eq:control:A3}
\end{eqnarray}
where the arguments ``$t-s+s'$'' are not written in the first equation for simplicity.

\paragraph{Study of $B$.}
We have $B=B_{1}+B_{2}$ with
\begin{equation*}
\begin{cases}
&B_{1}(t):= \int_{t}^{+\infty} [\Psi (s,Y_{1}(t)+f_{0}(t))-\Psi (s,Y_{2}(t)+f_{0}(t))] \\
&\hphantom{B_{1}(t):= \int_{t}^{+\infty} [f}  u^{\rm in}(s-t) \exp\left( -\int_{0}^{t} \Psi(s-t+t', Y_{1}(t') +f_{0}(t')) dt'\right)ds\\
&B_{2}(t):= \int_{t}^{+\infty} \Psi (s,Y_{2}(t)+f_{0}(t)) u^{\rm in}(s-t) \\
&\hphantom{B_{2}(t):= \int_{t}^{+\infty} f} [e^{-\int_{0}^{t} \Psi(s-t+t', Y_{1}(t') +f_{0}(t')) dt'}- e^{-\int_{0}^{t} \Psi(s-t+t', Y_{2}(t') +f_{0}(t')) dt'}]ds.
\end{cases}
\end{equation*}

- \underline{Study of $B_{1}$.}
Using the Lipschitz continuity of $\Psi$ and the fact that the argument in the exponential is non-positive, we have
\begin{equation}\label{eq:control:B1}
|B_{1}(t)|  \leq  {\rm Lip}(\Psi) \int_{t}^{+\infty} |Y_{1}(t) - Y_{2}(t)| u^{\rm in}(s-t)ds \leq {\rm Lip}(\Psi) || Y_{1}-Y_{2} ||_{L^{\infty}([0,t])},
\end{equation}
where we used that $\int_{0}^{+\infty} u^{\rm in}(s)ds =1$.

- \underline{Study of $B_{2}$.}
As for $A_{3}$, we have
\begin{equation}\label{eq:control:B2}
|B_{2}(t)| \leq || \Psi ||_{\infty} {\rm Lip}(\Psi) \, t\, || Y_{1}-Y_{2} ||_{L^{\infty}([0,t])}, 
\end{equation}
where we used once again that $\int_{0}^{+\infty} u^{\rm in}(s)ds =1$.

Gathering \eqref{eq:control:A1}, \eqref{eq:control:A2}, \eqref{eq:control:A3}, \eqref{eq:control:B1} and \eqref{eq:control:B2}, we get that there exists a non-decreasing function $g$ such that for all $t>0$,
\begin{equation*}
| \tilde{u}(t) |\leq g(t) || Y_{1}-Y_{2} ||_{L^{\infty}([0,t])}+ || \Psi ||_{\infty} \int_{0}^{t} |\tilde{u}(z)|dz.
\end{equation*}
Then, Lemma \ref{lem:Gronwall:Picard} gives that for all $T\geq 0$,
\begin{equation*}
|| \tilde{u} ||_{L^{\infty}([0,T])} \leq C_{T} \sup_{t\in [0,T]} \left\{ g(t) || Y_{1}-Y_{2} ||_{L^{\infty}([0,t])} \right\},
\end{equation*}
with $C_{T}$ being a non-decreasing function of $T$. Then, \eqref{eq:control:A:B} follows since $g$ is non-decreasing.

\subsection{Proof of Proposition \ref{prop:limit:equation:well:posed:delta=0}}
\label{sec:proof:prop:limit:equation}
\ \nsp\nsp

-$(i)$ Suppose that $(\overline{N}_t)_{t\geq 0}$ is a solution of \eqref{eq:limit:equation:counting:process}. Then, the mean cumulative intensity $\overline{\Lambda}(t)=\mathbb{E}[\overline{N}_t]$ is a non-decreasing locally bounded function satisfying
\begin{equation} \label{eq:solution:expectation}
\overline{\Lambda}(t) = \int_0^t \Psi_{0}\left( \int_0^{t'} h(t'-z) d\overline{\Lambda}_{z} + f_{0}(t')\right)dt'\;\;\text{for every}\;\;t\geq 0.
\end{equation}
By Lemma \ref{lem:solution:expectation:uniqueness}, we know that~\eqref{eq:solution:expectation} admits a unique solution which is furthermore of class $\mathcal{C}^{1}$ and we denote $\overline{\lambda}$ its derivative. Thus, we have $\mathbb{E}\left[ \overline{N}_{+}(dt) \right]=\overline{\Lambda}'(t)dt=\overline{\lambda}(t)dt$. 

\medskip
-$(ii)$ The first equation of \eqref{eq:limit:equation:counting:process:rewritten:delta=0} is a classical thinning equation so its solution $(\overline{N}_{t})_{t\geq 0}$ is a measurable function of $\Pi$ hence it is unique (once $\Pi$ is fixed).

To conclude  this step, it suffices to check that $(\overline{N}_{t})_{t\geq 0}$ satisfies the second equation of \eqref{eq:limit:equation:counting:process:rewritten:delta=0} where we remind that $\overline{\lambda}$ is the derivative of $\overline{\Lambda}$ which is the unique solution of \eqref{eq:solution:expectation}. But $\mathbb{E}[\overline{N}_t]=\int_{0}^{t}  \Psi_{0}(\int_0^{t'} h(t'-z) \overline{\lambda}(z) dz +f_{0}(t'))dt'$, which is equal to $\overline{\Lambda}(t)=\int_{0}^{t} \overline{\lambda}(t') dt'$ since $\overline{\Lambda}$ is a solution of \eqref{eq:solution:expectation}.\\

Finally, the two remaining points are rather simple. 
First, taking the derivative of \eqref{eq:solution:expectation} gives that $\overline{\lambda}(t)=\Psi_{0}\left(  \int_{0}^{t-} h(t-z) \overline{\lambda}(z)dz + f_{0}(t) \right) $.
Secondly, the solution of \eqref{eq:limit:equation:counting:process:rewritten:delta=0} is clearly a solution of \eqref{eq:limit:equation:counting:process} and $(i)$ tells that a solution of \eqref{eq:limit:equation:counting:process} is necessarily a solution of \eqref{eq:limit:equation:counting:process:rewritten:delta=0} which gives uniqueness.

\subsection{Proof of \eqref{eq:bound:delta:both:cases}}
\label{sec:proof:thm:counting:process:convergence}

For simplicity of notations in \eqref{eq:definition:coupling:ADRHP} and \eqref{eq:definition:coupling:limit:equation}, let us denote for all $t\geq 0$,
\begin{equation}\label{eq:def:gamma}
\left\{
\begin{aligned}
&\gamma^{n,i}_{t}:= \frac{1}{n}  \sum_{j=1}^n \left( \int_{0}^{t-} H_{ij}(t-z)N^{n,j}_{+}(dz) + F_{ij}(t) \right) , \\
&\overline{\gamma}(t) :=  \int_0^{t} m_{\mu_{H}}(t-z)\overline{\lambda}(z) dz +m_{\nu_{F}}(t),
\end{aligned}
\right.
\end{equation}
where $\overline{\lambda}$ is defined either in Proposition \ref{prop:limit:equation:well:posed:bounded:intensity} (under \hyperref[ass:H1]{$\mathcal{H}_{1}$}) or \ref{prop:limit:equation:well:posed:delta=0} (under \hyperref[ass:H2]{$\mathcal{H}_{2}$}). We have $\lambda^{n,i}_{t}=\Psi(S^{n,i}_{t-},\gamma^{n,i}_{t})$ and $ \overline{\lambda}^{i}_{t}=\Psi(\overline{S}^{i}_{t-},\overline{\gamma}(t))$. Notice that $\overline{\gamma}$ is a deterministic function (which does not depend on $i$) whereas $\gamma^{n,i}$ is random.

\paragraph{First point of \eqref{eq:bound:delta:both:cases}.} Assume that \hyperref[ass:H1]{$\mathcal{H}_{1}$} is satisfied. Then one can use the decomposition, 
\begin{equation}\label{eq:decomposition:delta:age:process:exact}
1=\mathds{1}_{\{S^{n,1}_{t-}=\overline{S}^{1}_{t-}\}} + \mathds{1}_{\{S^{n,1}_{t-}\neq \overline{S}^{1}_{t-}\}}
\end{equation}
and the fact that $| \lambda^{n,1}_{t} - \overline{\lambda}^{1}_{t} |\leq ||\Psi||_{\infty}$ to deduce from \eqref{eq:delta:n:intensity} that
\begin{equation}\label{eq:decomposition:delta:age:process}
\delta_{n}(\theta)\leq  \int_{0}^{\theta} \mathbb{E}\left[\Big| \lambda^{n,1}_{t} - \overline{\lambda}^{1}_{t} \Big| \mathds{1}_{\{S^{n,1}_{t-}=\overline{S}^{1}_{t-}\}}   \right] dt + || \Psi ||_{\infty} \int_{0}^{\theta} \mathbb{P}\left(S^{n,1}_{t-}\neq\overline{S}^{1}_{t-}  \right) dt.
\end{equation}
Let us denote 
$$
\left\{
\begin{aligned}
&A^{n}(\theta):=\int_{0}^{\theta} \mathbb{E}\left[\Big| \lambda^{n,1}_{t} - \overline{\lambda}^{1}_{t} \Big| \mathds{1}_{\{S^{n,1}_{t-}=\overline{S}^{1}_{t-}\}}   \right] dt, \\
&D^{n}(\theta):=\int_{0}^{\theta} \mathbb{P}\left(S^{n,1}_{t-}\neq\overline{S}^{1}_{t-}  \right) dt.
\end{aligned}
\right.
$$
\subparagraph{Study of $A^{n}(\theta)$.}
Using the Lipschitz continuity of $\Psi$, it is clear that for all $t$ in $[0,\theta]$, if $S^{n,1}_{t-}=\overline{S}^{1}_{t-}$, then
$
| \lambda^{n,1}_{t} - \overline{\lambda}^{1}_{t} | \leq {\rm Lip}(\Psi) | \gamma^{n,1}_{t} - \overline{\gamma}(t) |.
$
So, one deduces that
\begin{equation}\label{eq:decomposition:delta:ABC}
A^{n}(\theta) \leq {\rm Lip}(\Psi)(B^{n}_{1}(\theta)+B^{n}_{2}(\theta)+B^{n}_{3}(\theta)+C^{n}(\theta)),
\end{equation}
where
\begin{equation}\label{eq:definition:ABC}
\!\!\!\!\!\!\!\!\!\!\!\!\!\!\!\!
\left\{ 
\begin{array}{l}
\displaystyle B^{n}_{1}(\theta):=\int_{0}^{\theta} \mathbb{E}\Big[\Big|\int_0^{t} \frac{1}{n}\sum_{j=1}^n H_{1j}(t-z)[N^{n,j}_{+}(dz) - \overline{N}^{j}_{+}(dz)]\Big| \Big]  dt,\\
\displaystyle B^{n}_{2}(\theta):= \int_{0}^{\theta} \mathbb{E}\Big[\Big| \int_0^{t} \frac{1}{n}\sum_{j=1}^n H_{1j}(t-z) [\overline{N}^{j}_{+}(dz) - \overline{\lambda}^{j}_{z} dz]\Big| \Big]  dt,\\
\displaystyle B^{n}_{3}(\theta):= \int_{0}^{\theta} \mathbb{E}\Big[\Big| \int_0^{t} \frac{1}{n}\sum_{j=1}^n \left( H_{1j}(t-z) \overline{\lambda}^{j}_{z} - m_{\mu_{H}}(t-z)\overline{\lambda}(z) \right) dz\Big| \Big]  dt,\\
\displaystyle C^{n}(\theta):= \int_{0}^{\theta} \mathbb{E}\Big[\Big| \frac{1}{n}\sum_{j=1}^n F_{1j}(t) - m_{\nu_{F}}(t)\Big|  \Big]  dt.
\end{array}
\right.
\end{equation}

- \underline{Study of $B^{n}_{1}$.} Firstly, using (\hyperref[ass:mu:H:infty]{$\mathcal{A}^{\mu_{H}}_{\infty}$}) and then Lemma \ref{lem:double:Fubini}, we have
\begin{eqnarray}
B^{n}_{1}(\theta) &\leq & \int_{0}^{\theta} \mathbb{E}\Big[\int_0^{t} \frac{1}{n}\sum_{j=1}^n \Big|H_{1j}(t-z)\Big| \Big|N^{n,j}_{+}(dz) - \overline{N}^{j}_{+}(dz)\Big| \Big]  dt \nonumber\\
& \leq & \int_{0}^{\theta} \mathbb{E}\Big[ \int_0^{t} \frac{1}{n}\sum_{j=1}^n M_{\mu_{H}}(t-z)  \Big|N^{n,j}_{+}(dz) - \overline{N}^{j}_{+}(dz)\Big| \Big]  dt \nonumber\\ 
& = & \int_{0}^{\theta} \frac{1}{n}\sum_{j=1}^n \mathbb{E}\Big[\int_0^{t} M_{\mu_{H}}(t-z) d\Delta^{j}_{n}(z)\Big] dt \nonumber\\
& = & \int_{0}^{\theta} M_{\mu_{H}}(\theta-z) \delta_n(z) dz, \label{eq:bound:gronwall}
\end{eqnarray}
where the $\Delta^{j}_{n}$'s are given in \eqref{eq:def:Delta:i}.

- \underline{Study of $B^{n}_{2}$.} Secondly, 
using Cauchy-Schwartz inequality, we have
\begin{eqnarray}
B^{n}_{2}(\theta) & \leq &  \frac{1}{n} \int_{0}^{\theta} \mathbb{E}\Big[ \Big| \sum_{j=1}^n \int_0^{t} H_{1j}(t-z) [\overline{N}^{j}_{+}(dz) - \overline{\lambda}^{j}_{z} dz ] \Big|^{2} \Big]^{1/2} dt \nonumber \\
& = & \frac{1}{n} \int_{0}^{\theta} \mathbb{E}\Big[ \sum_{j=1}^n \int_0^{t} H_{1j}(t-z)^{2} \overline{\lambda}^{j}_{z} dz \Big]^{1/2} dt \nonumber \\
& \leq &  \frac{1}{n} \int_{0}^{\theta}  \Big(  n \int_0^{t} M_{\mu_{H}}(t-z)^{2}  || \Psi ||_{\infty} dz \Big)^{1/2} dt \nonumber \\
& = & \!\!\!\! \frac{1}{\sqrt{n}} || \Psi ||_{\infty}^{1/2} \!\! \int_{0}^{\theta} \!\!  \left(  \int_0^{t} M_{\mu_{H}}(t-z)^{2}  dz \right)^{1/2} \!\!\!\! dt:=\frac{1}{\sqrt{n}} \tilde{B}_{2}(\theta), \label{eq:bound:variance:mean:field:process}
\end{eqnarray}
by computing the bracket of a sum over a compensated point process (see \cite[Proposition II.4.1.]{gill_1997}), using (\hyperref[ass:mu:H:infty]{$\mathcal{A}^{\mu_{H}}_{\infty}$}) and the fact that $\overline{\lambda}^{j}$ is bounded by $|| \Psi ||_{\infty}$.

- \underline{Study of $B^{n}_{3}$.} Let us fix some $t$ in $[0,\theta]$ and $z$ in $[0,t]$ and denote $Y_{j}:=H_{1j}(t-z)\overline{\lambda}^{j}_{z}$. Since $\overline{\lambda}^{j}_{z}$ is the intensity of a solution of the limit equation, it is independent of $H_{1j}$. Moreover the $H_{1j}$'s are independent (see \eqref{eq:indep:Hij}) and the $\overline{\lambda}^{j}$'s are independent so the $Y_{j}$'s are independent. Hence,
\begin{equation*}
\mathbb{E}\left[ Y_{j} \right] = \mathbb{E}\left[ H_{1j}(t-z) \right] \mathbb{E}\left[ \overline{\lambda}^{j}_{z} \right] = m_{\mu_{H}}(t-z)\overline{\lambda}(z),
\end{equation*}
and
\begin{equation*}
\mathbb{V}\text{ar} \left( Y_{j} \right) = \mathbb{V}\text{ar} \left( H_{1j}(t-z) \right)\mathbb{V}\text{ar} \left( \overline{\lambda}^{j}_{z}  \right) + \mathbb{V}\text{ar} \left( H_{1j}(t-z) \right)\overline{\lambda}(z)^{2} 
+m_{\mu_{H}}(t-z)^{2}\mathbb{V}\text{ar} \left( \overline{\lambda}^{j}_{z}  \right).
\end{equation*}
Then, it follows from (\hyperref[ass:mu:H:infty]{$\mathcal{A}^{\mu_{H}}_{\infty}$}) and (\hyperref[ass:Psi:uniformly:bounded]{$\mathcal{A}^{\Psi}_{\infty}$}) that $\mathbb{V}\text{ar} \left( Y_{j} \right) \leq  3 M_{\mu_{H}}(t-z)^{2}|| \Psi ||_{\infty}^{2}$
since $m_{\mu_{H}}$ is dominated by $M_{\mu_{H}}$ and $\overline{\lambda}$ is bounded by $|| \Psi ||_{\infty}$. So, from the definition of $B^{n}_{3}$, using Cauchy-Schwartz inequality and the fact that the $Y_{j}$'s are independent, one has
\begin{eqnarray}
B^{n}_{3}(\theta) &\leq &\int_{0}^{\theta} \int_{0}^{t} \left( \frac{3}{n} M_{\mu_{H}}(t-z)^{2}|| \Psi ||_{\infty}^{2} \right)^{1/2} dz dt \nonumber\\
&\leq &\frac{\sqrt{3}}{\sqrt{n}} || \Psi ||_{\infty} \int_{0}^{\theta} \int_{0}^{t} M_{\mu_{H}}(t-z)  dz dt:= \frac{1}{\sqrt{n}} \tilde{B}_{3}(\theta) \label{eq:bound:product:variance}
\end{eqnarray}

- \underline{Study of $C^{n}$.} Since for all $t\geq 0$, $F_{11}(t),\dots ,F_{1n}(t)$ are i.i.d. random variables (see \eqref{eq:indep:Fij}) with expectation $m_{\nu_{F}}(t)$ and variance $V_{\nu_{F}}(t)$, one deduces from Cauchy-Schwartz inequality that
\begin{equation}\label{eq:bound:F0:variance}
C^{n}(\theta) \leq \frac{1}{\sqrt{n}} \int_{0}^{\theta} V_{\nu_{F}}(t)^{1/2} dt := \frac{1}{\sqrt{n}} \tilde{C}(\theta).
\end{equation}

Finally, one deduces from \eqref{eq:decomposition:delta:ABC}, \eqref{eq:bound:gronwall}, \eqref{eq:bound:variance:mean:field:process}, \eqref{eq:bound:product:variance} and \eqref{eq:bound:F0:variance} that
\begin{equation}\label{eq:gamma:n:final:bound}
A^{n}(\theta)\leq \frac{{\rm Lip}(\Psi)}{\sqrt{n}} \left( \tilde{B}_{2}(\theta) + \tilde{B}_{3}(\theta) + \tilde{C}(\theta) \right) + {\rm Lip}(\Psi) \int_{0}^{\theta} M_{\mu_{H}}(\theta-z) \delta_n(z) dz.
\end{equation}

\subparagraph{Study of $D^{n}(\theta)$.}
Since the initial conditions of $N^{n,1}$ and $\overline{N}^{1}$ are the same (equal to $N^{1}_{-}$), it holds that $S^{n,1}_{0}=\overline{S}^{1}_{0}$ a.s. If, at a fixed time $t\geq 0$, $S^{n,1}_{t-}\neq \overline{S}^{1}_{t-}$, then there is at least one point between $0$ and $t$ which is not common to both $N^{n,1}_{+}$ and $\overline{N}^{1}_{+}$, that is $\sup_{t'\in [0,t]} |N^{n,i}_{t'} - \overline{N}^{i}_{t'}| \neq 0$, hence
$
\mathbb{P}( S^{n,1}_{t-}\neq \overline{S}^{1}_{t-} )\leq \mathbb{P}( \sup_{t'\in [0,t]} |N^{n,i}_{t'} - \overline{N}^{i}_{t'}| \neq 0 ).
$
Moreover, since counting processes are piecewise constant with jumps of height $1$ a.s., it is clear that
\begin{equation*}
\mathbb{P}\Big( \sup_{t'\in [0,t]} |N^{n,1}_{t'} - \overline{N}^{1}_{t'}| \neq 0 \Big) \leq \mathbb{E}\Big[ \sup_{t'\in [0,t]} | N^{n,1}_{t'} -\overline{N}^1_{t'} | \Big],
\end{equation*}
where we used Markov's inequality. Using \eqref{eq:bound:L1:coupling}, one has $\mathbb{P}\left( S^{n,1}_{t-}\neq \overline{S}^{1}_{t-} \right) \leq \delta_{n}(t)$ and so
\begin{equation}\label{eq:bound:age:different:term}
D^{n}(\theta) \leq \int_{0}^{\theta} \delta_{n}(t) dt.
\end{equation}\\

Rewriting \eqref{eq:decomposition:delta:age:process} under the form
$
\delta_{n}(\theta)\leq A^{n}(\theta) + || \Psi ||_{\infty} D^{n}(\theta)
$,
one deduces from \eqref{eq:gamma:n:final:bound} and \eqref{eq:bound:age:different:term} that
\begin{equation}\label{eq:bound:final:delta:>0}
\delta_{n}(\theta)\leq \frac{{\rm Lip}(\Psi)}{\sqrt{n}} \left( \tilde{B}_{2}(\theta) + \tilde{B}_{3}(\theta) + \tilde{C}(\theta) \right) 
+  \int_{0}^{\theta} \left[ || \Psi ||_{\infty} + {\rm Lip}(\Psi) M_{\mu_{H}}(\theta-z)\right] \delta_n(z) dz
\end{equation}
where $\tilde{B}_{2}$, $\tilde{B}_{3}$, $\tilde{C}$ are respectively defined in \eqref{eq:bound:variance:mean:field:process}, \eqref{eq:bound:product:variance} and \eqref{eq:bound:F0:variance}. 
Since $M_{\mu_{H}}$ is locally square integrable,  $\theta\mapsto \tilde{B}_{2}(\theta)$ is locally bounded; since $M_{\mu_{H}}$ is locally integrable,  $\theta\mapsto \tilde{B}_{3}(\theta)$ is locally bounded; since $V_{\nu_{F}}$ is locally square root integrable,  $\theta\mapsto \tilde{C}(\theta)$ is locally bounded. Hence we proved the first point of \eqref{eq:bound:delta:both:cases}.

\paragraph{Second point of \eqref{eq:bound:delta:both:cases}.} Assume that \hyperref[ass:H2]{$\mathcal{H}_{2}$} is satisfied. Then the decomposition \eqref{eq:decomposition:delta:age:process:exact} is not helpful anymore. Whatever the age processes are, one always has $ \lambda^{n,i}_{t} =  \Psi_{0}(\gamma^{n,i}_{t})$ and 
$\overline{\lambda}^{i}_{t} = \Psi_{0}(\overline{\gamma}(t))$ for all $i=1,\dots ,n$ and $t> 0$. Remark that in this case, the intensities $\overline{\lambda}^{i}$ of the limit processes defined by \eqref{eq:definition:coupling:limit:equation} are deterministic and equal to $\overline{\lambda}$ defined in Proposition \ref{prop:limit:equation:well:posed:delta=0}. Instead of \eqref{eq:decomposition:delta:age:process} one should start from
$
\delta_{n}(\theta) =  \int_{0}^{\theta} \mathbb{E}[| \lambda^{n,1}_{t} - \overline{\lambda}^{1}_{t} |  ] dt .
$
One can prove in the same way as above that
\begin{equation}\label{eq:decomposition:delta:ABC:delta=0}
\delta_{n}(\theta) \leq {\rm Lip}(\Psi_{0}) (B^{n}_{1}(\theta)+B^{n}_{2}(\theta)+B^{n}_{3}(\theta)+C^{n}(\theta)),
\end{equation}
where $B^{n}_{1}$, $B^{n}_{2}$, $B^{n}_{3}$ and $C^{n}$ are defined by \eqref{eq:definition:ABC}. Since the uniform boundedness of $\Psi$ was not used in the study of $B^{n}_{1}$ and $C^{n}$ then \eqref{eq:bound:gronwall} and \eqref{eq:bound:F0:variance} still hold. It remains to control $B^{n}_{2}$ and $B^{n}_{3}$ under Assumption (\hyperref[ass:Psi:=Psi0]{$\mathcal{A}_{\Psi=\Psi_{0}}$}).

- \underline{Study of $B^{n}_{2}$.} Firstly, remind that for all $j$, $\overline{\lambda}^{j}_{t}=\overline{\lambda}(t)$ so,
using Cauchy-Schwartz inequality, we have
\begin{eqnarray}
B^{n}_{2}(\theta) & \leq &  \frac{1}{n} \int_{0}^{\theta} \mathbb{E}\Big[ \Big| \sum_{j=1}^n \int_0^{t} H_{1j}(t-z) [\overline{N}^{j}_{+}(dz) - \overline{\lambda}(z) dz ] \Big|^{2} \Big]^{1/2} dt \nonumber \\
& = & \frac{1}{n} \int_{0}^{\theta} \mathbb{E}\Big[ \sum_{j=1}^n \int_0^{t} H_{1j}(t-z)^{2} \overline{\lambda}(z) dz \Big]^{1/2} dt \nonumber \\
& \leq &  \frac{1}{\sqrt{n}} \int_{0}^{\theta}  \left(  \int_0^{t} M_{\mu_{H}}(t-z)^{2}  \overline{\lambda}(z) dz \right)^{1/2} dt:=\frac{1}{\sqrt{n}} \overline{B}_{2}(\theta), \label{eq:bound:variance:mean:field:process:delta=0}
\end{eqnarray}
by computing the bracket of a sum over a compensated point process (see \cite[Proposition II.4.1.]{gill_1997}) and then using (\hyperref[ass:mu:H:infty]{$\mathcal{A}^{\mu_{H}}_{\infty}$}). 

- \underline{Study of $B^{n}_{3}$.} Secondly, since $\overline{\lambda}^{j}_{t}=\overline{\lambda}(t)$ for all $j$, then $B^{n}_{3}$ rewrites as 
\begin{multline*}
B^{n}_{3}(\theta)=\int_{0}^{\theta} \mathbb{E}\Big[\Big| \int_0^{t} \frac{1}{n}\sum_{j=1}^n \left( H_{1j}(t-z)  - m_{\mu_{H}}(t-z) \right)\overline{\lambda}(z) dz\Big| \Big]  dt \\ 
\leq \int_{0}^{\theta}  \int_0^{t} \mathbb{E}\Big[\Big| \frac{1}{n}\sum_{j=1}^n \left( H_{1j}(t-z)  - m_{\mu_{H}}(t-z) \right) \Big|\Big]\overline{\lambda}(z) dz   dt.
\end{multline*}
Using Cauchy-Schwartz inequality, the fact that the $H_{1j}$'s are i.i.d. with mean function $m_{\mu_{H}}$ and that  for all $s\geq 0$, $\mathbb{V}\text{ar}(H_{1j}(s))\leq M_{\mu_{H}}(s)^{2}$ (which follows from (\hyperref[ass:mu:H:infty]{$\mathcal{A}^{\mu_{H}}_{\infty}$})), we have
\begin{equation}\label{eq:bound:product:variance:delta=0}
B^{n}_{3}(\theta)\leq \frac{1}{\sqrt{n}} \int_{0}^{\theta}  \int_0^{t}  M_{\mu_{H}}(t-z) \overline{\lambda}(z) dz dt := \frac{1}{\sqrt{n}} \overline{B}_{3}(\theta).
\end{equation}

Finally, one deduces from \eqref{eq:decomposition:delta:ABC:delta=0}, \eqref{eq:bound:gronwall}, \eqref{eq:bound:variance:mean:field:process:delta=0}, \eqref{eq:bound:product:variance:delta=0} and \eqref{eq:bound:F0:variance} that
\begin{equation}\label{eq:bound:final:delta:=0}
\delta_{n}(\theta)\leq \frac{{\rm Lip}(\Psi_{0})}{\sqrt{n}} \left( \overline{B}_{2}(\theta) + \overline{B}_{3}(\theta) + \tilde{C}(\theta) \right) +  \int_{0}^{\theta}  {\rm Lip}(\Psi_{0}) M_{\mu_{H}}(\theta-z) \delta_n(z) dz,
\end{equation}
where $\overline{B}_{2}$, $\overline{B}_{3}$, $\tilde{C}$ are respectively defined in \eqref{eq:bound:variance:mean:field:process:delta=0}, \eqref{eq:bound:product:variance:delta=0} and \eqref{eq:bound:F0:variance}. 

Since $M_{\mu_{H}}$ is locally square integrable and $\overline{\lambda}$ is continuous, $\theta\mapsto \overline{B}_{2}(\theta)$ is locally bounded; since $M_{\mu_{H}}$ is locally integrable and $\overline{\lambda}$ is continuous, $\theta\mapsto \overline{B}_{3}(\theta)$ is locally bounded; since $V_{\nu_{F}}$ is locally square root integrable, $\theta\mapsto \tilde{C}(\theta)$ is locally bounded. Hence we proved the second point of \eqref{eq:bound:delta:both:cases}.

\subsection{Proof of Proposition \ref{prop:linear:bound}}
\label{sec:proof:prop:linear:bound}
\ \nsp\nsp

-1. Assume that \hyperref[ass:H1]{$\mathcal{H}_{1}$} is satisfied. The mean intensity $\overline{\lambda}$ defined in Proposition \ref{prop:limit:equation:well:posed:bounded:intensity} is clearly uniformly bounded by $|| \Psi ||_{\infty}$.

Looking at \eqref{eq:bound:final:delta:>0}, one wants to find some uniform bounds on $\tilde{B}_{2}$, $\tilde{B}_{3}$, $\tilde{C}$ respectively defined in \eqref{eq:bound:variance:mean:field:process}, \eqref{eq:bound:product:variance} and \eqref{eq:bound:F0:variance}.
Firstly, since $M_{\mu_{H}}$ is square integrable and $\overline{\lambda}$ is uniformly bounded, it is clear from~\eqref{eq:bound:variance:mean:field:process} that 
\begin{equation}\label{eq:bound:variance:mean:field:process:enhanced:delta>0}
\tilde{B}_{2}(\theta)\leq || M_{\mu_{H}} ||_{2} || \Psi ||_{\infty}^{1/2} \theta.
\end{equation}
Secondly, using the integrability of $M_{\mu_{H}}$ and the boundedness of $\overline{\lambda}$, one deduces from~\eqref{eq:bound:product:variance} that
\begin{equation}\label{eq:bound:product:variance:enhanced:delta>0}
\tilde{B}_{3}(\theta)\leq \sqrt{3} || M_{\mu_{H}} ||_{1} || \Psi ||_{\infty}\theta.
\end{equation}
Finally, since $V_{\nu_{F}}$ is uniformly bounded, one deduces from \eqref{eq:bound:F0:variance} that
\begin{equation}\label{eq:bound:F0:variance:enhanced:delta>0}
\tilde{C}(\theta)\leq || V_{\nu_{F}} ||_{\infty}^{1/2} \theta.
\end{equation}
Using the fact that $\delta_{n}$ is a non-decreasing function, we find (with $\alpha={\rm Lip}(\Psi) || M_{\mu_{H}} ||_{1}$)
\begin{equation}\label{eq:bound:gronwall:enhanced:delta>0}
\begin{cases}
\int_{0}^{\theta}  {\rm Lip}(\Psi) M_{\mu_{H}}(\theta-z) \delta_n(z) dz \leq \alpha \delta_{n}(\theta)\\
\int_{0}^{\theta}  || \Psi ||_{\infty} \delta_n(z) dz \leq || \Psi ||_{\infty} \theta \delta_{n}(\theta).
\end{cases}
\end{equation}
Gathering \eqref{eq:bound:variance:mean:field:process:enhanced:delta>0}, \eqref{eq:bound:product:variance:enhanced:delta>0}, \eqref{eq:bound:F0:variance:enhanced:delta>0} and \eqref{eq:bound:gronwall:enhanced:delta>0} one deduces from \eqref{eq:bound:final:delta:>0} that $\delta_{n}(\theta)\leq \beta(\Psi,\mu_{H},\nu_{F}) \theta n^{-1/2}$ with
\begin{equation}\label{eq:explicit:bound:beta:delta>0}
\beta(\Psi,\mu_{H},\nu_{F}) := \frac{{\rm Lip}(\Psi)}{1-(\alpha+|| \Psi ||_{\infty} \theta)} \left( || M_{\mu_{H}} ||_{2} || \Psi ||_{\infty}^{1/2}  \right. \\
\left. + \sqrt{3} || M_{\mu_{H}} ||_{1} || \Psi ||_{\infty} + || V_{\nu_{F}} ||_{\infty}^{1/2} \right),
\end{equation}
as soon as $\alpha + || \Psi ||_{\infty} \theta<1$, i.e.  $\theta < (1-\alpha)/|| \Psi ||_{\infty}$. 

\medskip
-2. Assume \hyperref[ass:H2]{$\mathcal{H}_{2}$} is satisfied. Under the assumptions of Proposition \ref{prop:linear:bound}, the mean intensity $\overline{\lambda}$ defined in Proposition \ref{prop:limit:equation:well:posed:delta=0} is uniformly bounded thanks to Lemma \ref{lem:boundedness:mean:intensity}.

Looking at \eqref{eq:bound:final:delta:=0}, one wants to find some uniform bounds on $\overline{B}_{2}$, $\overline{B}_{3}$, $\tilde{C}$ respectively defined in \eqref{eq:bound:variance:mean:field:process:delta=0}, \eqref{eq:bound:product:variance:delta=0} and \eqref{eq:bound:F0:variance}. In the same way as above, one can deduce from \eqref{eq:bound:variance:mean:field:process:delta=0}, \eqref{eq:bound:product:variance:delta=0} and \eqref{eq:bound:F0:variance} that 
\begin{equation}\label{eq:bound:enhanced}
\begin{cases}
\overline{B}_{2}(\theta)\leq || M_{\mu_{H}} ||_{2} || \overline{\lambda} ||_{\infty}^{1/2} \theta\\
\overline{B}_{3}(\theta)\leq || M_{\mu_{H}} ||_{1} || \overline{\lambda} ||_{\infty}\theta\\
\tilde{C}(\theta)\leq || V_{\nu_{F}} ||_{\infty}^{1/2} \theta.
\end{cases}
\end{equation}
Using \eqref{eq:bound:enhanced} and the first equation of \eqref{eq:bound:gronwall:enhanced:delta>0} which is still valid is this case, one deduces from \eqref{eq:bound:final:delta:=0} that
\begin{equation*}
\delta_{n}(\theta)\leq \frac{{\rm Lip}(\Psi)}{\sqrt{n}} \left( || M_{\mu_{H}} ||_{2} || \overline{\lambda} ||_{\infty}^{1/2}  +  || M_{\mu_{H}} ||_{1} || \overline{\lambda} ||_{\infty} + || V_{\nu_{F}} ||_{\infty}^{1/2} \right) \theta + \alpha\delta_{n}(\theta),
\end{equation*}
which leads to $\delta_{n}(\theta)\leq \beta(\Psi,\mu_{H},\nu_{F}) \theta n^{-1/2}$ with
\begin{equation}\label{eq:explicit:bound:beta}
\beta(\Psi,\mu_{H},\nu_{F}) := \frac{{\rm Lip}(\Psi)}{1-\alpha} \left( || M_{\mu_{H}} ||_{2} || \overline{\lambda} ||_{\infty}^{1/2}  +  || M_{\mu_{H}} ||_{1} || \overline{\lambda} ||_{\infty} + || V_{\nu_{F}} ||_{\infty}^{1/2} \right),
\end{equation}
for every $\theta\geq 0$. Notice that an explicit expression of $|| \overline{\lambda} ||_{\infty}$ with respect to $M_{\mu_{H}}$, $m_{\nu_{F}}$ and $\Psi$ can be obtained thanks to Lemma \ref{lem:boundedness:mean:intensity}.

\section{Lemmas}

\subsection{Point processes}
Here we collect some technical lemmas about point processes. 

The following lemma is used to show the well-posedness of the studied point processes.

\begin{lem}\label{lem:stochastic:domination}
Let $n\geq 1$ be an integer, let $(g_{i})_{i=1,\dots ,n}$ be a family of locally integrable functions, $(a_{i})_{i=1,\dots ,n}$ be a family of non-negative real numbers and $h:\mathbb{R}_{+}\to \mathbb{R}$ be a locally integrable function. Let $(\Pi^{i}(dt,dx))_{i\geq 1}$ be some i.i.d. $\mathbb{F}$-Poisson measures with intensity $1$ on $\mathbb{R}_{+}^{2}$.

Let $(N_t^i)^{i=1,..,n}_{t \geq 0}$ be a family of counting processes such that for $i=1,..,n$ and all $t\geq 0$ ,
\begin{equation}\label{eq:Hawkes:Thinning:dominated}
N_t^i= \int_0^t \int_0^\infty \mathds{1}_{\displaystyle \{ x\leq \lambda^{i}_{t'} \}} \, \Pi^i(dt',dx),
\end{equation}
where the $\lambda^{i}$'s are $\mathbb{F}$-predictable processes such that 
$$\lambda^{i}_{t}\leq g_{i}(t) + a_{i} \frac{1}{n}\sum_{j=1}^{n} \int_{0}^{t-} |h(t-z)| N^{j}(dz).$$ 
Then, the linear multivariate Hawkes process $(\tilde{N}_t^i)^{i=1,..,n}_{t \geq 0}$ defined by
\begin{equation}\label{eq:Hawkes:Thinning:dominating}
\tilde{N}_t^i= \int_0^t \int_0^\infty \mathds{1}_{\displaystyle \Big\{ x\leq g_{i}(t') + a_{i} \frac{1}{n}\sum_{j=1}^{n} \int_{0}^{t'-} |h(t'-z)|\tilde{N}^{j}(dz) \Big\}} \, \Pi^i(dt',dx),
\end{equation}
is such that for all $i=1,\dots ,n$, $\tilde{N}^{i}$ stochastically dominates $N^{i}$ in the sense that $N^{i}\subset \tilde{N}^{i}$ where $N^{i}$ (resp. $\tilde{N}^{i}$) denotes the point process associated with the counting process $(N^{i}_{t})_{t\geq 0}$ (resp. $(\tilde{N}^{i}_{t})_{t\geq 0}$). In particular, the processes $(N_t^i)^{i=1,..,n}_{t \geq 0}$ are well-defined.
\end{lem}
\begin{proof}
First, let us note that the processes $(\tilde{N}_t^i)^{i=1,..,n}_{t \geq 0}$ are well-defined by the Galton-Watson representation of the linear Hawkes process introduced in \cite{hawkes_1974} when the $g_{i}$'s are constant in time (see \cite[Proposition B.3]{chevallier2015microscopic} when the $g_{i}$'s are more generally locally integrable functions).

We are going to show by induction that 
$$\forall t\geq 0,\ \lambda^{i}_{t}\leq \tilde{\lambda}^{i}_{t}:=g_{i}(t) + a_{i} \frac{1}{n}\sum_{j=1}^{n} \int_{0}^{t-} |h(t-u)|\tilde{N}^{j}(du).$$ 
Indeed, for all time $t$ less than the first point of either $N:=\cup_{i=1}^{n} N^{i}$ or $\tilde{N}:=\cup_{i=1}^{n} \tilde{N}^{i}$, the respective intensities are such that $\lambda^{i}_{t} \leq g_{i}(t) = \tilde{\lambda}^{i}_{t}$.
Hence, the first point of $N\cup\tilde{N}$ (denoted by $T_{1}$) is a point of $\tilde{N}$ (and possibly a point of $N$). Let us denote $(T_{k})_{k\geq 1}$ the ordered sequence of the points of $N\cup\tilde{N}$.

Let us fix some $k_{0}\geq 1$. Suppose that for every $i=1,\dots ,n$, $\lambda^{i}_{t} \leq  \tilde{\lambda}^{i}_{t}$, for all $t\leq T_{k_{0}}$. Then, it is clear that for every $k=1,\dots ,k_{0}$, $T_{k}\in \tilde{N}$, hence for every $i=1,\dots ,n$, $\tilde{N}^{i}$ stochastically dominates $N^{i}$ up to time $T_{k_{0}+1-}$. Moreover, it implies that for every $i=1,\dots ,n$ and for all $t$ in $(T_{k_{0}},T_{k_{0}+1}]$,
\begin{eqnarray*}
\lambda^{i}_{t} &\leq &g_{i}(t) + a_{i} \frac{1}{n}\sum_{j=1}^{n} \int_{0}^{t-} |h(t-z)|N^{j}(dz)\\
 &\leq &g_{i}(t) + a_{i} \frac{1}{n}\sum_{j=1}^{n} \int_{0}^{t-} |h(t-z)|\tilde{N}^{j}(dz) = \tilde{\lambda}^{i}_{t},
\end{eqnarray*}
since $|h|$ is a non negative function.
Therefore, by induction on $k$, the desired stochastic domination holds true for all time. In particular, the dominated processes are well-defined.
\end{proof}

\begin{lem}\label{lem:continuity:age:process}
If $N$ admits the bounded $\mathbb{F}$-intensity $\lambda_{t}$ and $(S_{t-})_{t\geq 0}$ denote its associated predictable age process, then the distribution of $S_{t-}$ denoted by $w_{t}$ is such that $t\mapsto w_{t}$ belongs to $\mathcal{BC}(\mathbb{R}_{+},\mathcal{P}(\mathbb{R}_{+}))$.
\end{lem}
\begin{proof}
This continuity result comes from the fact that the probability that $N$ has a point in an interval goes to $0$ as the size of the interval goes to $0$. Indeed, let $t,t'$ be positive real numbers, $\mathbb{P}\left( \overline{N}([t,t+t'))\neq 0 \right)\leq \mathbb{E}\left[ \overline{N}([t,t+t')) \right] = \mathbb{E}[ \int_{t}^{t+t'} \lambda_{z} dz ]$ goes to $0$ as $t'$ goes to $0$. Moreover, $S_{(t+t')-}=S_{t-}+t'$ as soon as there is no point of $N$ in the interval $[t,t+t')$ and so one has (reminding that $\tilde{W}_{1}$ denotes the modified Wasserstein distance defined in \eqref{eq:def:modified:Wasserstein:distance})
\begin{equation*}
\tilde{W}_{1}(w_{t+t'},w_{t})\leq \mathbb{E}\left[ \min\left( \left| S_{(t+t')-}-S_{t-} \right|, 1\right) \right]\leq t' + \mathbb{P}\left( \overline{N}([t,t+t'))\neq 0 \right),
\end{equation*}
which goes to $0$ as $t'$ goes to $0$. The same argument for $t'<0$ gives continuity.
\end{proof}

\subsection{Analytic lemmas}

Here, we collect some analytic lemmas regarding the convolution equations used throughout the present paper. First, here are two lemmas introduced in \cite{delattre2016}. 

The first one is an easy application of Fubini's Theorem (see \cite[Lemma 22]{delattre2016} for a proof).

\begin{lem}\label{lem:double:Fubini}Let $\Phi:\mathbb{R}_{+}\rightarrow\mathbb{R}$
be locally integrable and let $\alpha:\mathbb{R}_{+}\rightarrow\mathbb{R}$
have bounded variations on compact intervals, satisfying $\alpha\left(0\right)=0$.
Then, for all $t\geq0$,
$$
\int_{0}^{t}\int_{0}^{s}\Phi\left(s-u\right)d\alpha\left(u\right)ds=\int_{0}^{t}\Phi\left(t-s\right)\alpha\left(s\right)ds,
$$
where the integral has to be understood in the Stieltjes' sense.
\end{lem}

The second one is a rather classical generalization of Gr\"onwall Lemma (see \cite[Lemma 23]{delattre2016} for a proof).

\begin{lem}\label{lem:Gronwall:Picard}
Let $\Phi:\mathbb{R}_{+}\rightarrow\mathbb{R}_{+}$
be locally integrable and $g:\mathbb{R}_{+}\rightarrow\mathbb{R}_{+}$
be locally bounded.
Consider a locally bounded non-negative function $u$ such that for
all $t\geq0$, $u_{t}\leq g_{t}+\int_{0}^{t}\Phi\left(t-s\right)u_{s}ds$.
Then, $\sup_{t\in \left[0,T\right]}u_{t}\leq C_{T}\sup_{t\in \left[0,T\right]}g_{t}$,
for some constant $C_{T}$ depending only on $T>0$ and $\Phi$. Moreover, $C_{T}$ can be taken as a non-decreasing function of $T$.
\end{lem}

Note that we added to the first statement that $C_{T}$ can be taken as a non-decreasing function of $T$. It is not given in \cite[Lemma 23]{delattre2016} but it is a direct consequence of the proof.

Then, here is a well-posedness result which is a generalization of \cite[Lemma 24]{delattre2016}.

\begin{lem}\label{lem:solution:expectation:uniqueness}
Let $\Phi :\mathbb{R}\rightarrow \mathbb{R}_{+}$ be Lipschitz-continuous, $h : \mathbb{R}_{+} \rightarrow\mathbb{R}$ be locally integrable and $f_{0}: \mathbb{R}_{+}\rightarrow \mathbb{R}$ be continuous. The equation
\begin{equation} \label{eq:expectation:counting:process}
m_t = \int_0^t \Phi\left(\int_0^{t'} h(t'-z)dm_z + f_{0}(t')\right)dt'
\end{equation}
has a unique locally bounded solution. Furthermore, $m$ is of class $\mathcal{C}^1$
on $\mathbb{R}_{+}$.
\end{lem}
\begin{proof}
The proof is similar to \cite[Lemma 24]{delattre2016}. We refrain from reproducing it here; instead, we only indicate the minor changes that are required to make it fit the current framework, i.e. the addition of the function $f_{0}$.
The ``uniqueness'' part is exactly the same.
The ``existence'' part requires $f_{0}$ to be locally integrable in order to have locally boundedness in the Picard iteration.
Finally, we need $f_{0}$ to be continuous to show by induction that at each step of the Picard iteration the function is $\mathcal{C}^{1}$ on $\mathbb{R}_{+}$ and so it is for the limit, that is the solution of~\eqref{eq:expectation:counting:process}.
\end{proof}

Here is given an analytic result which is used to give a uniform bound on the mean intensity of a Hawkes process under stationary conditions.

\begin{lem}\label{lem:boundedness:mean:intensity}
Let $\Phi :\mathbb{R}\rightarrow \mathbb{R}_{+}$ be Lipschitz-continuous and $h : \mathbb{R}_{+} \rightarrow\mathbb{R}$ be integrable such that ${\rm Lip}(\Phi) || h ||_{1}<1$. Moreover, let $f_{0}: \mathbb{R}_{+}\rightarrow \mathbb{R}$ be uniformly bounded. 

If $g:\mathbb{R}_{+}\rightarrow \mathbb{R}_{+}$ is a locally bounded function satisfying
\begin{equation}
g(t) \leq \Phi\left( \int_{0}^{t} h(t-u) g(u)du + f_{0}(t) \right)
\end{equation}
for every $t> 0$, then $g$ is uniformly upper bounded by
$$
M:=\frac{\Phi(0) + {\rm Lip}(\Phi) || f_{0} ||_{\infty}}{1-{\rm Lip}(\Phi) || h ||_{1}}.
$$
\end{lem}
\begin{proof}
For any $t> 0$,
$$
g(t) \leq \Phi(0) + {\rm Lip}(\Phi) \left( \int_{0}^{t} |h(t-u)| g(u)du + || f_{0} ||_{\infty} \right)
$$
hence, thanks to the locally boundedness of $g$, for every $T\geq 0$,
$$
\sup_{t\in [0,T]} g(t) \leq \Phi(0) + {\rm Lip}(\Phi) \left( || h ||_{1} \sup_{t\in [0,T]} g(t) + || f_{0} ||_{\infty} \right)
$$
and
$$
\sup_{t\in [0,T]} g(t) \leq \frac{1}{1-{\rm Lip}(\Phi) || h ||_{1}} \left[\Phi(0) + {\rm Lip}(\Phi) || f_{0} ||_{\infty} \right]=M.
$$
\end{proof}

\newpage

\bibliographystyle{abbrv}
{\small \bibliography{references}}

\end{document}